\documentclass[12pt]{amsart}

\usepackage[utf8]{inputenc}
\usepackage[T1]{fontenc}
\usepackage[english]{babel}

\usepackage{amsfonts}
\usepackage{amsmath}
\usepackage{amssymb}
\usepackage{amstext}
\usepackage{amsthm}

\usepackage{fullpage}

\usepackage[shortlabels]{enumitem}
\usepackage{capt-of}
\usepackage{hyperref}

\usepackage{tikz}
\usetikzlibrary{arrows,backgrounds,chains,decorations.pathreplacing,patterns,shapes, calc, positioning}

\allowdisplaybreaks

\newcommand{\SSYT}{\operatorname{SSYT}}
\newcommand{\wt}{\operatorname{wt}}
\newcommand{\area}{\operatorname{area}}
\newcommand{\dinv}{\operatorname{dinv}}
\newcommand{\tdinv}{\operatorname{tdinv}}
\newcommand{\pdinv}{\operatorname{pdinv}}
\newcommand{\maxtdinv}{\operatorname{maxtdinv}}
\newcommand{\comp}{\operatorname{comp}}
\newcommand{\PF}{\operatorname{PF}}
\newcommand{\Pref}{\operatorname{Pr}\nolimits}
\newcommand{\touch}{\operatorname{touch}}
\newtheorem{theorem}{Theorem}[section]
\newtheorem{lemma}[theorem]{Lemma}
\newtheorem{proposition}[theorem]{Proposition}
\newtheorem{corollary}[theorem]{Corollary}

\theoremstyle{definition}
\newtheorem{definition}[theorem]{Definition}
\newtheorem{example}[theorem]{Example}

\theoremstyle{remark}
\newtheorem{remark}[theorem]{Remark}

\title{The $(km,kn)$-square paths theorem}

\author{Dun Qiu}
\address{Center for Combinatorics, LPMC, Nankai University, Tianjin 300071, P. R. China}
\email{qiudun@nankai.edu.cn}
\thanks{Dun Qiu is supported in part by the National Natural Science Foundation of China (No. 12671390),
the Natural Science Foundation of Tianjin (No. 24JCZDJC01390), and the Fundamental Research Funds for the Central Universities (No. 105-63263093).
The authors thank Alessandro Iraci for his helpful remarks that clarified the sign conventions in the statements of Theorem~\ref{thmmain} and Theorem \ref{cor:fall-decorated}.}

\author{Minhao Zhang}
\address{Center for Combinatorics, LPMC, Nankai University, Tianjin 300071, P. R. China}
\email{carpentercheungmath@gmail.com}
\thanks{}

\date{\today}

\subjclass[2020]{Primary 05E05; Secondary 05A19, 05E10.}
\keywords{Symmetric functions, square paths theorem, $(km,kn)$-preference functions,
elliptic Hall algebra, row strict tableaux}

\begin{document}

\begin{abstract}
For positive integers $k,m,n$ with $\gcd(m,n)=1$, we prove a
$(km,kn)$-square paths theorem: the symmetric function
$(-1)^{k-1}\frac{[km]_q}{[k]_q}p_k[-MX^{m,n}]\cdot1$ is a weighted sum over
$(km,kn)$-preference functions. This proves a corrected form of the conjecture
of Iraci, Pagaria, Paolini, and Vanden Wyngaerd and generalizes the square paths
theorem of Sergel. We encode preference functions by row strict tableaux and
closed walks on a set of states, compare them with parking functions for
each value of the touch statistic,
and apply Mellit's compositional $(km,kn)$-shuffle theorem. Finally, we
establish a fall Delta square theorem conjectured by Iraci, Pagaria, and Paolini.
\end{abstract}

\maketitle

\section{Introduction}
The \emph{nabla operator} $\nabla$, introduced by Bergeron and Garsia
\cite{BG99} and studied further by Bergeron, Garsia, Haiman, and
Tesler~\cite{BGHT99}, acts diagonally on the modified Macdonald polynomials in
$\Lambda$. It is closely tied to the space of \emph{diagonal
harmonics}, a bigraded $S_n$-module studied by Haiman~\cite{Hai94}. Garsia and Haiman \cite{GH96}
conjectured that its bigraded Frobenius characteristic is $\nabla e_n$,
where $e_n$ is the $n$th elementary symmetric function. Haiman proved
this conjecture~\cite{Hai02}.

The \emph{shuffle theorem}, conjectured by Haglund, Haiman, Loehr,
Remmel, and Ulyanov \cite{HHLRU05}, expresses $\nabla e_n$ as a
weighted sum over parking functions. Grouping the terms by their
underlying Dyck paths also gives a positive sum of LLT polynomials.
Haglund, Morse, and Zabrocki~\cite{HMZ12} introduced the symmetric
functions $C_\alpha[X;q]$ and formulated the compositional shuffle
conjecture. Carlsson and Mellit~\cite{CM18} proved that conjecture
and, consequently, the shuffle theorem.

Gorsky and Negu\textcommabelow{t} \cite{GN} formulated the rational shuffle
conjecture. Bergeron, Garsia, Leven, and Xin~\cite{BGLX16}
then proposed the compositional $(km,kn)$-shuffle conjecture, which
gives a parking function formula for
$C_\alpha[-MX^{m,n}]\cdot1$, where $M=(1-q)(1-t)$. The operators arise from the elliptic
Hall algebra action studied by Feigin and Tsymbaliuk~\cite{FT} and
Schiffmann and Vasserot~\cite{SV}. In the reformulation of
Blasiak et al.~\cite[Lemma~3.4.1]{path},
\[
C_\alpha[-MX^{m,1}]\cdot 1=\nabla^m C_\alpha[X;q].
\]
Mellit~\cite{Mel21} proved the compositional $(km,kn)$-shuffle
conjecture.

Loehr and Warrington~\cite{LW08} conjectured a combinatorial
formula for $\nabla^m s_\lambda$, proved by Blasiak et al.~\cite{dens}.
Earlier, Loehr and
Warrington~\cite{LW07} had conjectured a preference function formula
for $(-1)^{n-1}\nabla p_n$. Sergel~\cite{Ser17} proved this
conjecture, establishing the square paths theorem.

\begin{theorem}[{Sergel~\cite[Theorem~4.5]{Ser17}}]
For every positive integer $n$,
\[
(-1)^{n-1}\nabla p_n
=
\sum_{\pi\in\Pref_{n,n}}
q^{\dinv(\pi)}t^{\area(\pi)}x_\pi.
\]
\end{theorem}

Iraci, Pagaria, Paolini, and Vanden Wyngaerd~\cite[Conjecture~4.2]{IPPVW23}
conjectured, in our notation, that for all positive integers $k,m,n$ with
$\gcd(m,n)=1$,
\[
\frac{[km]_q}{[k]_q}\,p_{km,kn}
=\sum_{\pi\in\operatorname{LRP}(km,kn)}q^{\dinv(\pi)}t^{\area(\pi)}x_\pi,
\]
where $[r]_q:=(1-q^r)/(1-q)$, $p_{km,kn}$ is their rational analogue of $\nabla p_n$, 
and $\operatorname{LRP}(km,kn)$, their set of labeled rectangular
paths, is our $\Pref_{km,kn}$.
They~\cite[Theorem~6.1]{IPPVW23} proved it when $k=1$.

The conjecture needs a sign correction, which is made explicit in
Section~\ref{subsec:schiffmann}. In the notation of that section, the
corrected statement is as follows.

\begin{theorem}[= Theorem~\ref{thmpn}]\label{thmmain}
For all positive integers $k,m,n$ with $\gcd(m,n)=1$,
\begin{equation}
(-1)^{k-1}\frac{[km]_q}{[k]_q}\,
p_k[-MX^{m,n}]\cdot 1
=
\sum_{\pi\in\Pref_{km,kn}}
q^{\dinv(\pi)}t^{\area(\pi)}x_\pi.
\end{equation}
\end{theorem}

When $m=n=1$, Theorem~\ref{thmmain} specializes to Sergel's
square paths theorem.

The central step is the following refinement of
Theorem~\ref{thmmain} by the number of touches.
\begin{theorem}[= Theorem~\ref{thmEnkPr}]
   For positive integers $k,m,n$ with $\gcd(m,n)=1$ and
   $1\leq\ell\leq k$, we have
  \[
\frac{[km]_q}{[\ell]_q}E_{k,\ell}[-MX^{m,n}]\cdot1
=
\sum_{\substack{\pi\in\Pref_{km,kn}\\\touch(\pi)=\ell}}
q^{\dinv(\pi)}t^{\area(\pi)}x_\pi.
\]
\end{theorem}
Mellit~\cite[Section~6]{Mel21} proved the compositional
$(km,kn)$-shuffle theorem (Theorem~\ref{kmknshuffle}), which identifies
$E_{k,\ell}[-MX^{m,n}]\cdot1$ with the corresponding sum over
parking functions (Theorem~\ref{thmEnkPF}). Thus the combinatorial
part of the proof is the comparison
\[
[\ell]_q
\sum_{\substack{\pi\in\Pref_{km,kn}\\\touch(\pi)=\ell}}
q^{\dinv(\pi)}t^{\area(\pi)}x_\pi
=
[km]_q
\sum_{\substack{\pi\in\PF_{km,kn}\\\touch(\pi)=\ell}}
q^{\dinv(\pi)}t^{\area(\pi)}x_\pi.
\tag{\ref{eqmaincomb}}
\]
Once this is established, the expansion of the power sum by Can and
Loehr~\cite{can} (Proposition~\ref{propptoEnk}) gives
Theorem~\ref{thmmain} by summing over $\ell$.

We also prove the fall Delta square theorem conjectured by Iraci,
Pagaria, and Paolini~\cite[Conjecture~6.9]{IPP26}. It is a fall version
of the Delta square conjecture of D'Adderio, Iraci, and Vanden
Wyngaerd~\cite{DIVWsq}, the square paths analogue of the Delta conjecture
of Haglund, Remmel, and Wilson~\cite{HRW18}. In the course of
the proof, we establish a corrected form of their operator identity
\cite[Conjecture~6.10]{IPP26}, stated below as
Proposition~\ref{propDelta}. We write
$\operatorname{LRP}(w,h)_{\ast d}$ for their labeled paths from
$(0,0)$ to $(w,h)$, ending with an East step and carrying $d$
decorated falls; a \emph{fall} is an East step that is immediately
followed by another East step. We use the area and dinv statistics of
Iraci et al.~\cite[Section~3]{IPP26}; for the square paths in the next
theorem, they are recalled in Section~\ref{subsec:delta}.
\begin{theorem}\label{thmDelta}
  For $k\geq1$ and $d\geq0$, we have
  \begin{equation}\label{eqDelta}
    (-1)^{k-1}\Theta_{e_d}\nabla p_k
    =\sum_{\pi\in\operatorname{LRP}(k+d,k+d)_{\ast d}}
      q^{\dinv(\pi)}t^{\area(\pi)}x_\pi.
  \end{equation}
\end{theorem}

The proof has two parts. First, the combinatorial skewing argument of
Iraci et al.~\cite[Section~5.2]{IPP26}, applied to
Theorem~\ref{thmmain} with $(m,n)=(1,d+1)$, expresses the sum in
\eqref{eqDelta} as a Schur skewing of
$p_k[-MX^{1,d+1}]\cdot1$. We then prove Proposition~\ref{propDelta},
the corrected form of their operator identity~\cite[Conjecture~6.10]{IPP26},
by expanding the power sum in
Negu\textcommabelow{t} elements and applying the skewing and Theta identities of
D'Adderio, Interdonato, Iraci, and Pagaria
\cite[Theorems~4.9 and 5.14]{DIIP26}. This identifies the skewed
expression with the left side of \eqref{eqDelta}.

The paper is organized as follows. Section~\ref{sec:prelim} fixes the conventions for symmetric functions and combinatorial objects used
in this paper. Section~\ref{sec:rst} associates a row strict tableau to
each preference function and identifies its area and dinv weights
(Proposition~\ref{propdinvh}). In Section~\ref{sec:main} we encode the tableaux
by closed walks on a set of states. Comparing walks that use the same
letters but start at different states proves
Proposition~\ref{propqcyc}, from which \eqref{eqmaincomb} follows.
Section~\ref{sec:main} also proves Theorem~\ref{thmDelta} by combining the
combinatorial skewing argument with the Negu\textcommabelow{t} element calculation.

\section{Preliminaries}\label{sec:prelim}
Section~\ref{subsec:symfun} fixes our conventions for symmetric functions and
plethystic substitution. For more details, see Macdonald~\cite{Mac95},
Haglund~\cite{Hag08}, and D'Adderio, Iraci, and Vanden
Wyngaerd~\cite{DIVW21}.
Section~\ref{subsec:comb} recalls $(km,kn)$-preference functions and their
statistics, following Iraci et al.~\cite{IPPVW23}.
Unlike the parking functions in Mellit's theorem, the preference
functions considered here may pass below the main diagonal.
Section~\ref{subsec:schiffmann} recalls the Schiffmann algebra action in the
reformulation of Blasiak et al.~\cite{Delta,path}.
Section~\ref{subsec:negut} introduces the Negu\textcommabelow{t} elements in the notation of
Blasiak et al. and records the skewing and Theta identities of
D'Adderio, Interdonato, Iraci, and Pagaria~\cite{DIIP26} needed for
Theorem~\ref{thmDelta}.

\subsection{Symmetric functions}\label{subsec:symfun}
A \emph{partition} is a weakly decreasing sequence
$\lambda=(\lambda_1,\ldots,\lambda_\ell)$ of positive integers.
Its size and length are
$|\lambda|:=\lambda_1+\cdots+\lambda_\ell$ and
$\ell(\lambda):=\ell$.
We write $\lambda\vdash n$ when $|\lambda|=n$, and denote the
conjugate partition by $\lambda'$. For integers $a,d\geq0$, we
write $(a)^d=(a^d)$ for the partition with $d$ parts equal to $a$;
we regard $(0)^d$ as the empty
partition, so $s_{(0)^d}=1$. A \emph{composition} of $n$
is a sequence of positive integers
$\alpha=(\alpha_1,\ldots,\alpha_\ell)$ with
$
|\alpha|:=\alpha_1+\cdots+\alpha_\ell=n.
$
We write $\alpha\vDash n$ and $\ell(\alpha):=\ell$.

Let $X=x_1+x_2+\cdots$, and let $\Lambda$ be the algebra of
symmetric functions in $X$ with coefficients in
$\mathbb Q(q,t)$. Its standard bases include the monomial,
complete homogeneous, elementary, power sum, and Schur bases,
denoted respectively by
$
\{m_\lambda\}_\lambda,
\{h_\lambda\}_\lambda,
\{e_\lambda\}_\lambda,
\{p_\lambda\}_\lambda,
\{s_\lambda\}_\lambda.
$
We use the standard algebra involution $\omega$ of $\Lambda$, given by
$\omega(p_r)=(-1)^{r-1}p_r$.
Let $\langle\cdot,\cdot\rangle$ be the Hall scalar product, for which
the Schur functions form an orthonormal basis.
For any $f\in\Lambda$, we denote by $f^\perp$ the adjoint of
multiplication by $f$ with respect to $\langle\cdot,\cdot\rangle$.
We recall the \emph{Jacobi--Trudi formula} and its dual form,
\begin{equation}\label{eqJT}
  s_\lambda = \det[h_{\lambda_i-i+j}]_{1\le i,j\le \ell(\lambda)},
  \qquad
  s_{\lambda'} = \det[e_{\lambda_i-i+j}]_{1\le i,j\le \ell(\lambda)}.
\end{equation}

We use plethystic substitution throughout the paper. For an expression $E$
in indeterminates, $p_r[E]$ is obtained by replacing each
indeterminate by its $r$th power and extending this operation to
sums and quotients. 
For $f\in\Lambda$, the expression $f[E]$ is obtained by writing
$f$ in the power sum basis and replacing each $p_r$ by $p_r[E]$.

Following Garsia and Haglund \cite{GHag}, define the symmetric
functions $E_{k,\ell}[X]$ by
\begin{equation}\label{eqdefEnk}
e_k\left[X\frac{1-z}{1-q}\right]
=
\sum_{\ell=1}^{k}
\frac{(1-z)(1-zq)\cdots(1-zq^{\ell-1})}
     {(1-q)(1-q^2)\cdots(1-q^\ell)}
E_{k,\ell}[X].
\end{equation}

Haglund, Morse, and Zabrocki~\cite{HMZ12} introduced the operators
$C_a(q)$, for $a\geq1$, by
\[
C_a(q)f[X]
:=
\left(-\frac{1}{q}\right)^{a-1}
\langle z^a\rangle
f\left[X-\frac{1-1/q}{z}\right]
\sum_{r\geq0}h_r[X]z^r,
\]
where $\langle z^a\rangle G(z)$ denotes the coefficient of
$z^a$ in the formal Laurent series $G(z)$. For a composition
$\alpha=(\alpha_1,\ldots,\alpha_\ell)$, set
\[
C_\alpha[X;q]
:=
C_{\alpha_1}(q)C_{\alpha_2}(q)\cdots
C_{\alpha_\ell}(q)(1).
\]
The functions $E_{k,\ell}$ collect the $C_\alpha$ with a
fixed number of parts.
\begin{proposition}[{Haglund, Morse, and Zabrocki~\cite[Corollary~5.5]{HMZ12}}]\label{propEtoC}
For $1\leq\ell\leq k$,
\begin{equation}\label{eqEtoC}
E_{k,\ell}
=
\sum_{\substack{\alpha\vDash k\\\ell(\alpha)=\ell}}
C_\alpha[X;q].
\end{equation}
\end{proposition}

\begin{proposition}[{Can and Loehr~\cite[Eq.~(1.1)]{can}}]\label{propptoEnk}
For every $k\geq1$,
\begin{equation}\label{eqptoEnk}
(-1)^{k-1}p_k
=
\sum_{\ell=1}^{k}
\frac{[k]_q}{[\ell]_q}E_{k,\ell},
\end{equation}
where $[r]_q=(1-q^r)/(1-q)$.
\end{proposition}

Throughout, set $M:=(1-q)(1-t)$. Let $\widetilde H_\mu[X;q,t]$ denote
the modified Macdonald polynomials of Garsia and Haiman~\cite{GH93}, normalized
as in Haglund~\cite{Hag08}:
\[
\widetilde H_\mu[X;q,t]
=t^{n(\mu)}J_\mu\Bigl[\frac{X}{1-t^{-1}};q,t^{-1}\Bigr],
\qquad
n(\mu):=\sum_{i\geq1}(i-1)\mu_i,
\]
where $J_\mu$ is the integral form of the Macdonald
polynomial~\cite{Mac95}; for example,
$\widetilde H_{(2)}=s_2+qs_{11}$. The $\widetilde H_\mu$ form a basis of $\Lambda$.
For a partition $\mu$, set
\[
B_\mu(q,t):=\sum_{(i,j)\in\mu}q^{i-1}t^{j-1},
\qquad
\overline {B_\mu(q,t)}:=B_\mu(q^{-1},t^{-1}),
\]
where $(i,j)$ runs over the cells of the diagram of $\mu$, with
$i$ the column index and $j$ the row index, so that row $j$ has
$\mu_j$ cells.
For $f\in\Lambda$, the notations $f[B]$ and
$f[\overline B]$ denote the operators determined by
\[
f[B]\widetilde H_\mu
=
f[B_\mu(q,t)]\widetilde H_\mu,
\qquad
f[\overline B]\widetilde H_\mu
=
f[\overline {B_\mu(q,t)}]\widetilde H_\mu.
\]

The nabla operator $\nabla$ is defined on the basis
$\{\widetilde H_\mu\}$ by
\[
\nabla\widetilde H_\mu[X;q,t]
= e_{|\mu|}[B] \widetilde H_\mu[X;q,t]=
q^{n(\mu')}t^{n(\mu)} \widetilde H_\mu[X;q,t].
\]

Set
\[
\Pi_\mu:=\prod_{(i,j)\in\mu\setminus\{(1,1)\}}
   (1-q^{i-1}t^{j-1}),\qquad
\mathbf{\Pi}\widetilde H_\mu[X;q,t]:=\Pi_\mu\widetilde H_\mu[X;q,t],
\]
with $\Pi_\emptyset=1$. Following D'Adderio, Iraci, and Vanden
Wyngaerd~\cite[Eq.~(28)]{DIVW21},
let $f$ and $F$ be homogeneous symmetric functions of degrees $a$ and
$b$, respectively. Then
\[
\Theta_f F:=
\begin{cases}
0,&a>0\text{ and }b=0,\\
fF,&a=b=0,\\
\mathbf{\Pi}\bigl(f[X/M]\,\mathbf{\Pi}^{-1}F\bigr),&b>0.
\end{cases}
\]
We extend this definition linearly in $f$ and $F$. In particular,
$\Theta_{e_0}$ is the identity and $\Theta_{e_d}1=0$ for $d>0$.

\subsection{Combinatorial objects}\label{subsec:comb}
We recall $(km,kn)$-paths and their statistics, following Iraci
et al.~\cite{IPPVW23}.
From now on, unless stated otherwise, $k,m,n$ denote positive integers
with $\gcd(m,n)=1$.
A \emph{$(km,kn)$-path} is a lattice path from $(0,0)$ to $(km,kn)$
consisting of unit North and East steps, abbreviated $N$ and $E$,
respectively, and ending with an $E$ step. A \emph{$(km,kn)$-Dyck path}
is a $(km,kn)$-path that stays weakly above the main diagonal $y=\frac{n}{m}x$.

Suppose that the $i$th North step of a $(km,kn)$-path $\pi$ begins at
$(x,y)$. Set
$
a_i:=\frac{m}{n}y-x.
$
Thus $a_i$ is its signed horizontal distance from the main diagonal,
and we call $na_i=my-nx$ the \emph{rank} of this North step.
The \emph{area word} of $\pi$ is $(a_1,\ldots,a_{kn})$. Set
$
s:=-\min_{1\leq i\leq kn}a_i.
$
Then $s=0$ if and only if $\pi$ is a $(km,kn)$-Dyck path. The line
$my=n(x-s)$ is the \emph{base diagonal} of $\pi$, as in the square case
treated by D'Adderio, Iraci, and Vanden Wyngaerd~\cite{DIVWsq}.

For a $(km,kn)$-path $\pi$, define its \emph{touch statistic} by
\[
\touch(\pi):=
\#\{1\le i\le kn:a_i+s=0\}.
\]
Thus $\touch(\pi)$ counts the North steps that begin on the base
diagonal. Since $\gcd(m,n)=1$, consecutive lattice points on the
base diagonal differ by $(m,n)$, while the North steps begin at the
distinct heights $0,1,\ldots,kn-1$. Hence at most $k$ North steps
begin on the base diagonal, and at least one does, since the minimum
defining $s$ is attained. Thus $1\le \touch(\pi)\le k$.

For a $(km,kn)$-Dyck path $\pi$, let its successive intersections with
the main diagonal be
\[
(j_0m,j_0n),\ldots,(j_\ell m,j_\ell n),
\qquad 0=j_0<j_1<\cdots<j_\ell=k.
\]
Define its \emph{touch composition} by
\[
\comp(\pi):=
(j_1-j_0,\ldots,j_\ell-j_{\ell-1})\vDash k.
\]
In this case $s=0$, and each intersection except $(km,kn)$
is the starting point of a North step. Hence
$
\touch(\pi)
=\#\{0\le j<k:a_{jn+1}=0\}
=\ell,
$
the number of parts of $\comp(\pi)$.

A \emph{$(km,kn)$-preference function} is a $(km,kn)$-path with a
positive integer label on each North step, such that the labels on
consecutive North steps increase strictly. If its underlying path
is a $(km,kn)$-Dyck path, it is a \emph{$(km,kn)$-parking function}. Set
\[
\Pref_{km,kn}:=\{\text{$(km,kn)$-preference functions}\},\qquad \PF_{km,kn}:=\{\text{$(km,kn)$-parking functions}\}.
\]
If $w_i$ is the label of the $i$th North step, set
$
x_\pi:=\prod_{i=1}^{kn}x_{w_i}.
$

For $\pi\in\Pref_{km,kn}$, its \emph{area} is
\[
\area(\pi):=\sum_{i=1}^{kn}\lfloor a_i+s\rfloor.
\]
This counts the unit squares between $\pi$ and its base diagonal.

Let $\mu(\pi)$ be the Ferrers diagram consisting of the cells above
$\pi$ inside the $(km)\times(kn)$ rectangle. For $c\in\mu(\pi)$, let
$a(c)$ and $\ell(c)$ be the numbers of cells of $\mu(\pi)$ strictly
to the right and strictly below $c$, respectively. The \emph{path
dinv} is
\[
\pdinv(\pi):=
\#\left\{
c\in\mu(\pi):
\frac{a(c)}{\ell(c)+1}\leq\frac{m}{n}
<\frac{a(c)+1}{\ell(c)}
\right\},
\]
where the right inequality holds by convention if $\ell(c)=0$.

For $\pi\in\Pref_{km,kn}$, the \emph{attacks} are the pairs in the
two disjoint sets below. Their total number is the \emph{maximal
temporary dinv}:
\[
\begin{aligned}
\maxtdinv(\pi):={}&
\#\bigl\{(i,j):1\leq i<j\leq kn,\ 0\leq a_j-a_i<m/n\bigr\}\\
&+\#\bigl\{(i,j):1\leq j<i\leq kn,\ 0<a_j-a_i\leq m/n\bigr\}.
\end{aligned}
\]
The \emph{temporary dinv} counts the attacks whose labels increase:
\[
\tdinv(\pi):=
\#\bigl\{(i,j):(i,j)\text{ is an attack of }\pi,\ w_i<w_j\bigr\}.
\]
For a fixed underlying path, $\maxtdinv(\pi)$ is the maximum of
$\tdinv$ over its valid labelings.

We also set
\[
B(\pi):=
\#\left\{1\leq i\leq kn:a_i<-\frac{m}{n}\right\}.
\]
It counts North steps whose upper endpoints lie strictly below the
main diagonal.

The \emph{dinv} statistic is
\[
\dinv(\pi):=
\tdinv(\pi)+\pdinv(\pi)-\maxtdinv(\pi)
-B(\pi).
\]
\begin{example}\label{examplePathStats}
Let $k=2$, $m=3$, and $n=4$. The path $\pi$ in
Figure~\ref{figPathStats} has vertical run lengths
$(2,1,0,1,3,1)$; its labels, read from bottom to top in each
column, are $(2,6),(3),\varnothing,(1),(4,5,7),(8)$.
\begin{center}
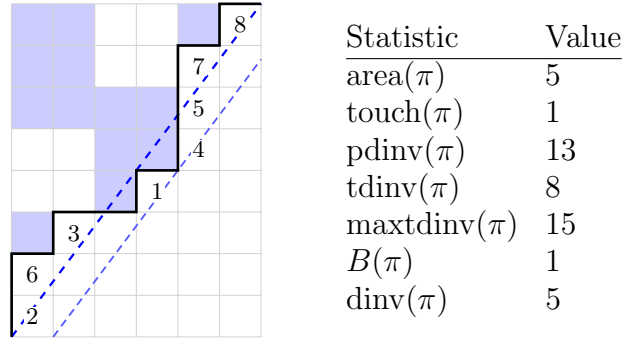

\begin{minipage}{\linewidth}
\centering
\begin{tikzpicture}[x=5.5mm,y=5.5mm]
  \fill[blue!20] (0,2) rectangle (1,3);
  \fill[blue!20] (0,5) rectangle (1,8);
  \fill[blue!20] (1,5) rectangle (2,8);
  \fill[blue!20] (2,3) rectangle (3,6);
  \fill[blue!20] (3,4) rectangle (4,6);
  \fill[blue!20] (4,7) rectangle (5,8);
  \draw[step=1,gray!35,very thin] (0,0) grid (6,8);
  \draw[blue,dashed,line width=.85pt] (0,0)--(6,8);
  \draw[blue!65,densely dashed,line width=.7pt] (1,0)--(6,6.6667);
  \draw[line width=1pt]
    (0,0)--(0,2)--(1,2)--(1,3)--(2,3)--(3,3)
    --(3,4)--(4,4)--(4,7)--(5,7)--(5,8)--(6,8);
  \foreach \x/\y/\z in
    {0/.5/2,0/1.5/6,1/2.5/3,3/3.5/1,
     4/4.5/4,4/5.5/5,4/6.5/7,5/7.5/8}
    \node[font=\scriptsize,fill=white,inner sep=.6pt] at ({\x+.5},\y) {$\z$};
  \node[anchor=west,inner sep=0pt]
    at ([xshift=11mm]current bounding box.east) {%
\begin{tabular}{@{}ll@{}}
Statistic & Value\\
\hline
$\area(\pi)$ & $5$\\
$\touch(\pi)$ & $1$\\
$\pdinv(\pi)$ & $13$\\
$\tdinv(\pi)$ & $8$\\
$\maxtdinv(\pi)$ & $15$\\
$B(\pi)$ & $1$\\
$\dinv(\pi)$ & $5$
\end{tabular}
  };
\end{tikzpicture}
\captionof{figure}{An example of a $(6,8)$-preference function and its statistics}\label{figPathStats}
\end{minipage}
\end{center}

In Figure~\ref{figPathStats}, the blue cells contribute to
$\pdinv(\pi)$.
The eight North steps have area word
\[
\left(0,\frac34,\frac12,-\frac34,-1,
      -\frac14,\frac12,\frac14\right).
\]
Thus $s=1$, and the contributions $\lfloor a_i+s\rfloor$ are
$(1,1,1,0,0,0,1,1)$, giving $\area(\pi)=5$. Only the fifth North
step begins on the base diagonal, so $\touch(\pi)=1$. Its upper
endpoint is $(4,5)$, strictly below the main diagonal, and no
other North step has this property; hence $B(\pi)=1$.

The Ferrers diagram above $\pi$ has $21$ cells, of which the
$13$ blue cells satisfy the path dinv inequalities. For the
blue cell with lower left corner $(0,5)$, the arm and leg both
have size $3$; the cell at $(0,3)$ has arm $2$ and
leg $1$, so it fails the first inequality. There are $6$ attacks
with $i<j$ and $9$ with $i>j$. All $6$ in the first group
have $w_i<w_j$, while in the second group only
$(i,j)=(4,1),(3,2)$ do. Consequently,
$\maxtdinv(\pi)=6+9=15$, $\tdinv(\pi)=6+2=8$, and
$\dinv(\pi)=8+13-15-1=5$.
\end{example}

For $\pi\in\PF_{km,kn}$, we have $B(\pi)=0$, and $\dinv(\pi)$
specializes to the dinv statistic for $(km,kn)$-parking functions of
Bergeron et al.~\cite[Algorithm~6.1]{BGLX16}.

When $m=n=1$, the area word is integral and
\[
\area(\pi)=\sum_{i=1}^{k}(a_i+s).
\]
An attack $(i,j)$ with $i<j$ has $a_i=a_j$, while one
with $i>j$ has $a_j=a_i+1$. The statistic $\tdinv(\pi)$
counts those attacks with $w_i<w_j$.
Moreover, specializing the definitions of Iraci
et al.~\cite[Definitions~3.16 and 3.19]{IPPVW23} to square paths gives
\[
\dinv(\pi)=\tdinv(\pi)+\#\{i:a_i<0\}.
\]
The attacks with $i<j$ and $i>j$ are Sergel's primary and
secondary dinv contributions, while $\#\{i:a_i<0\}$ is her
tertiary dinv~\cite{Ser17}, called the bonus dinv by D'Adderio, Iraci, and
Vanden Wyngaerd~\cite{DIVWsq}. Hence the
right side of Theorem~\ref{thmmain} specializes to the monomial
form of the identity proved by Sergel~\cite[Theorem~4.5]{Ser17};
her formulation uses standard labels and fundamental quasisymmetric
functions, and it is equivalent to this monomial form by standardization.

\begin{example}\label{exampleSquareSergel}
Take $k=5$ and $m=n=1$. The $(5,5)$-preference function
$\pi$ in Figure~\ref{figSquareSergel} has vertical run lengths $(0,3,0,1,1)$, with
labels $\varnothing,(1,3,5),\varnothing,(2),(4)$ in its five
columns. In particular, it passes below the main diagonal.
\begin{center}
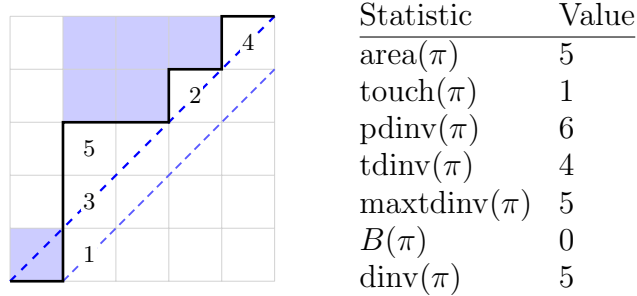

\begin{minipage}{\linewidth}
\centering
\begin{tikzpicture}[x=7mm,y=7mm]
  \fill[blue!20] (0,0) rectangle (1,1);
  \fill[blue!20] (1,3) rectangle (2,5);
  \fill[blue!20] (2,3) rectangle (3,5);
  \fill[blue!20] (3,4) rectangle (4,5);
  \draw[step=1,gray!35,very thin] (0,0) grid (5,5);
  \draw[blue,dashed,line width=.85pt] (0,0)--(5,5);
  \draw[blue!65,densely dashed,line width=.7pt] (1,0)--(5,4);
  \draw[line width=1pt]
    (0,0)--(1,0)--(1,3)--(2,3)--(3,3)
    --(3,4)--(4,4)--(4,5)--(5,5);
  \foreach \x/\y/\z in
    {1/.5/1,1/1.5/3,1/2.5/5,3/3.5/2,4/4.5/4}
    \node[font=\scriptsize,fill=white,inner sep=.6pt] at ({\x+.5},\y) {$\z$};
  \node[anchor=west,inner sep=0pt]
    at ([xshift=11mm]current bounding box.east) {%
\begin{tabular}{@{}ll@{}}
Statistic & Value\\
\hline
$\area(\pi)$ & $5$\\
$\touch(\pi)$ & $1$\\
$\pdinv(\pi)$ & $6$\\
$\tdinv(\pi)$ & $4$\\
$\maxtdinv(\pi)$ & $5$\\
$B(\pi)$ & $0$\\
$\dinv(\pi)$ & $5$
\end{tabular}
  };
\end{tikzpicture}

\captionof{figure}{An example of a $(5,5)$-preference function and its statistics}\label{figSquareSergel}
\end{minipage}
\end{center}

The blue cells in Figure~\ref{figSquareSergel} contribute to
$\pdinv(\pi)$.
Its area word is $(-1,0,1,0,0)$, so $s=1$,
$\area(\pi)=0+1+2+1+1=5$, and $\touch(\pi)=1$.
Of the ten cells above the path, the six blue cells satisfy the
path dinv inequalities. The attacks with $i<j$ are
$(2,4),(2,5),(4,5)$; those with $i>j$ are
$(4,3),(5,3)$. Thus $\maxtdinv(\pi)=3+2=5$.
The labels satisfy $w_i<w_j$ for two attacks in each group, so
$\tdinv(\pi)=2+2=4$.
No North step ends strictly below the main diagonal, hence
$B(\pi)=0$ and $\dinv(\pi)=4+6-5=5$.
The single negative area entry contributes one tertiary dinv in
Sergel's square paths statistic~\cite{Ser17}.
Thus Sergel's primary, secondary, and tertiary dinv give
$2+2+1=5$, the value of $\dinv(\pi)$ above.
\end{example}

\subsection{Schiffmann algebra}\label{subsec:schiffmann}
The elliptic Hall algebra of Burban and
Schiffmann~\cite{BS12}, denoted by $\mathcal E$, is also
called the \emph{Schiffmann algebra}. Its action on
$\Lambda$ was studied by Feigin and Tsymbaliuk~\cite{FT}
and Schiffmann and Vasserot~\cite{SV}. We use the
reformulation of Blasiak et al.~\cite{Delta,path}.

We write the operators in our combinatorial formulas as
$f[-MX^{m,n}]$; they are defined below. The dictionary between these
operators and the $Q$-operator construction is given by D'Adderio,
Interdonato, Iraci, and Pagaria~\cite[Remark~2.11]{DIIP26}.
With the operator definitions of Iraci et al.~\cite[Definitions~2.5 and~2.6]{IPPVW23},
the left side of their Conjecture~4.2 must be multiplied by
$(-1)^{kn-1}$.\footnote{\label{fnsign}With the same definitions, the rational shuffle
theorem as stated in \cite[Theorem~4.1]{IPPVW23} needs the sign
$(-1)^{k(n-1)}$ for $e_{km,kn}$. In the normalization in which that theorem
holds as stated, the factor above becomes $(-1)^{k-1}$, the sign in
Theorem~\ref{thmmain}.}
In our notation, the left side of the corrected identity is
$[km]_q/[k]_q$ times $(\omega p_k)[-MX^{m,n}]\cdot1$, and
$\omega p_k=(-1)^{k-1}p_k$; this gives Theorem~\ref{thmmain}.

The algebra $\mathcal E$ is generated by subalgebras
$\Lambda(X^{m,n})$, one for each coprime pair of integers
$(m,n)$, together with the central Laurent polynomial algebra
$
\mathbb Q(q,t)[c_1^{\pm1},c_2^{\pm1}].
$
Each $\Lambda(X^{m,n})$ is isomorphic to $\Lambda$, and for
$f\in\Lambda$ we write $f(X^{m,n})$ for the image of $f$.
See Blasiak et al.~\cite[Section~3.1]{Delta} for the defining relations.
We write $f[-MX^{m,n}]$ for the image of the plethystic substitution
$f[-MX]$ under the isomorphism $f\mapsto f(X^{m,n})$; for example,
\[
p_r[-MX^{m,n}]=-(1-q^r)(1-t^r)\,p_r(X^{m,n}).
\]
For $F\in\mathcal E$, we write $F\cdot1$ for the action of $F$
on $1\in\Lambda$ given by Proposition~\ref{prop:Eactonsym}.

\begin{proposition}[{Blasiak et al.~\cite[Proposition~3.3.1]{path}}]\label{prop:Eactonsym}
There is a unique action of $\mathcal E$ on $\Lambda$ with the
following properties:
\begin{enumerate}
    \item The central parameters act as scalars:
    \[
    c_1\mapsto1,\qquad c_2\mapsto(qt)^{-1}.
    \]

    \item For $f\in\Lambda$, the elements $f(X^{1,0})$ and
    $f(X^{-1,0})$ act as Macdonald eigenoperators:
    \[
    f(X^{1,0})\mapsto(\omega f)[B-1/M],
    \qquad
    f(X^{-1,0})\mapsto
    (\omega f)\bigl[\,\overline{1/M-B}\,\bigr].
    \]

    \item The elements $f(X^{0,1})$ and $f(X^{0,-1})$ act by
    \[
    f(X^{0,1})\mapsto f[-X/M]^\bullet,
    \qquad
    f(X^{0,-1})\mapsto f(X)^\perp.
    \]
\end{enumerate}
Here $f^\bullet$ denotes multiplication by $f$, and
$f^\perp$ is the adjoint defined in Section~\ref{subsec:symfun}.
In the second item, the bar replaces $q,t$ by $q^{-1},t^{-1}$;
thus $\overline{1/M-B}=qt/M-\overline B$.
\end{proposition}

\begin{lemma}[{Blasiak et al.~\cite[Lemma~3.4.1]{path}}]\label{lemma:nabla}
Conjugation by $\nabla$ satisfies
\[
\nabla f(X^{m,n})\nabla^{-1}
=
f(X^{m+n,n}).
\]
\end{lemma}

Proposition~\ref{prop:Eactonsym} gives
$f[-MX^{0,1}]\cdot1=f$. Applying
Lemma~\ref{lemma:nabla} repeatedly and using
$\nabla1=1$, we obtain, for every $f\in\Lambda$,
\[
f[-MX^{m,1}]\cdot1=\nabla^m f.
\]

The compositional $(km,kn)$-shuffle theorem was conjectured
by Bergeron et al.~\cite[Conjecture~3.3]{BGLX16} and proved by
Mellit~\cite[Section~6]{Mel21}. Using the dictionary of D'Adderio
et al.~\cite[Remark~2.11]{DIIP26} between the $Q$-operators and the
operators $f[-MX^{m,n}]$, it gives a combinatorial
interpretation of $C_\alpha[-MX^{m,n}]\cdot1$.

\begin{theorem}[{Mellit~\cite[Section~6]{Mel21}}]\label{kmknshuffle}
Let $k,m,n$ be positive integers with $\gcd(m,n)=1$. For
every $\alpha\vDash k$,
\[
C_\alpha[-MX^{m,n}]\cdot1
=
\sum_{\substack{\pi\in\PF_{km,kn}\\
                 \comp(\pi)=\alpha}}
q^{\dinv(\pi)}t^{\area(\pi)}x_\pi.
\]
\end{theorem}

The symmetric function considered in this paper is
$(-1)^{k-1}p_k[-MX^{m,n}]\cdot1$. When $n=1$,
\[
(-1)^{k-1}p_k[-MX^{m,1}]\cdot1
=
(-1)^{k-1}\nabla^m p_k.
\]
In particular, $m=1$ gives
$(-1)^{k-1}\nabla p_k$, which is the algebraic side of the square paths
theorem proved by Sergel \cite{Ser17}.

\subsection{\texorpdfstring{Negu\textcommabelow{t}}{Negut} elements and Theta operators}\label{subsec:negut}
We use the Negu\textcommabelow{t} elements introduced by
Negu\textcommabelow{t}~\cite[Proposition~6.1]{Negut}, in the reformulation of
Blasiak et al.~\cite[Section~3.6]{path}. The formulas below connect
their action on $1$ with the Theta operators of Section~\ref{subsec:symfun}.

Let $\mathcal E^+$ be the subalgebra of $\mathcal E$ generated by
$\Lambda(X^{a,b})$ for $a>0$. Write $\mathcal S$ for the shuffle algebra
and $\psi:\mathcal S\longrightarrow\mathcal E^+$ for the isomorphism
of Blasiak et al.~\cite[Proposition~3.5.1]{path}. Its degree $\ell$
part is represented by Laurent polynomials in $z_1,\ldots,z_\ell$,
modulo the kernel of the map $\mathbf H_{q,t}^{(\ell)}$, which is
defined from the Weyl symmetrization $\boldsymbol{\sigma}$ by
\begin{align*}
\boldsymbol{\sigma}(\phi)
&=
\sum_{\tau\in S_\ell}
\tau\!\left(
\frac{\phi(z)}
     {\prod_{1\leq a<b\leq\ell}(1-z_b/z_a)}
\right),\\
\mathbf H_{q,t}^{(\ell)}(\phi)
&=
\boldsymbol{\sigma}\!\left(
\phi(z)\prod_{1\leq a<b\leq\ell}
\frac{1-qt z_a/z_b}
     {(1-qz_a/z_b)(1-tz_a/z_b)}
\right).
\end{align*}
In particular,
$\psi(z_1^{c_1}\cdots z_\ell^{c_\ell})
=p_1[-MX^{1,c_1}]\cdots p_1[-MX^{1,c_\ell}]$.
For $\gamma=(\gamma_1,\ldots,\gamma_\ell)\in\mathbb Z^\ell$, put
\[
R_\gamma(z):=
\frac{z_1^{\gamma_1}\cdots z_\ell^{\gamma_\ell}}
     {\prod_{j=1}^{\ell-1}(1-qt z_j/z_{j+1})}.
\]
There is a Laurent polynomial $\mu_\gamma$ with
$\mathbf H_{q,t}^{(\ell)}(\mu_\gamma)
 =\mathbf H_{q,t}^{(\ell)}(R_\gamma)$.
The element $D_\gamma:=\psi(\mu_\gamma)\in\mathcal E^+$
is independent of this choice, since $\psi$ is defined on Laurent
polynomials modulo the kernel of $\mathbf H_{q,t}^{(\ell)}$; see
Blasiak et al.~\cite[Eqs.~(51)--(52)]{path}. We use $D_\gamma$ as an operator on
$\Lambda$ through Proposition~\ref{prop:Eactonsym}.

For the power sum of degree $k\geq1$, set
\[
\mathbf1=(1,\ldots,1)\in\mathbb Z^k,\qquad
\gamma^{(0)}=(1^k),\qquad
\gamma^{(i)}=(1^{i-1},2,1^{k-i-1},0)
\quad(1\leq i<k).
\]
Negu\textcommabelow{t}~\cite[Theorem~1.1, Eq.~(1.2)]{Negut}, together with
the reformulation of Blasiak et al.~\cite[Eq.~(42)]{tameLLT}
and the definition of $D_\gamma$ above~\cite[Eqs.~(51)--(52)]{path},
gives
\begin{equation}\label{eqnegut}
(-1)^{k-1}p_k[-MX^{1,d+1}]\cdot1
=
D_{\gamma^{(0)}+d\mathbf1}\cdot1+
\sum_{i=1}^{k-1}(qt)^{k-i}
D_{\gamma^{(i)}+d\mathbf1}\cdot1
\qquad(d\geq0).
\end{equation}

We also use the following skewing identity of
D'Adderio et al.~\cite[Theorem~4.9, Eq.~(14)]{DIIP26}:
for $j\geq0$ and $\alpha\in\mathbb Z^\ell$,
\begin{equation}\label{eqDIIP}
h_j^\perp D_\alpha
=
\sum_{\chi\in\{0,1\}^\ell}
D_{\alpha-\chi}h_{j-|\chi|}^\perp,
\end{equation}
where $h_a^\perp=0$ when $a<0$.

\begin{proposition}\label{propnegut}
For $k\geq1$, $d\geq0$, and $0\leq i<k$,
\[
s_{(k-1)^d}^\perp
\bigl(D_{\gamma^{(i)}+d\mathbf1}\cdot1\bigr)
=
\sum_{\beta\in\mathbb Z_{\geq0}^k, |\beta|=d}
D_{\gamma^{(i)}+\beta}\cdot1.
\]
\end{proposition}
\begin{proof}
Since $h_a^\perp1=0$ for $a>0$, applying \eqref{eqDIIP} to $1$
gives, for every $\alpha\in\mathbb Z^k$,
\[
h_j^\perp(D_\alpha\cdot1)
=
\sum_{\chi\in\{0,1\}^k,\ |\chi|=j}
D_{\alpha-\chi}\cdot1.
\]
For $z=(z_1,\ldots,z_k)$, the coefficient of $z^\chi$
in $e_j(z)$ is $1$ precisely when
$\chi\in\{0,1\}^k$ and $|\chi|=j$.
Iterating the preceding identity and using both formulas
in~\eqref{eqJT} therefore gives
\begin{equation}\label{eqschur}
s_\lambda^\perp(D_\alpha\cdot1)
=
\sum_{\mu\in\mathbb Z_{\geq0}^k}
[z^\mu]s_{\lambda'}(z)\,
D_{\alpha-\mu}\cdot1.
\end{equation}
For $\lambda=(k-1)^d$, its conjugate is $(d^{k-1})$, and
\[
s_{(d^{k-1})}(z)
=(z_1\cdots z_k)^d h_d(z_1^{-1},\ldots,z_k^{-1})
=
\sum_{\beta\in\mathbb Z_{\geq0}^k, |\beta|=d}
z^{d\mathbf1-\beta}.
\]
Substituting $\alpha=\gamma^{(i)}+d\mathbf1$ and
$\mu=d\mathbf1-\beta$ into \eqref{eqschur} proves the claim.
\end{proof}

To state the Theta identity we need the following word notation from
D'Adderio et al.~\cite[Definition~5.13]{DIIP26}.
For a composition $\alpha=(\alpha_1,\ldots,\alpha_\ell)$, let
\[
w(\alpha)=\ell^{\alpha_1}
(\ell-1)^{\alpha_2}\cdots1^{\alpha_\ell}.
\]
For a word $u$, let $\operatorname{Des}(u)$ be its \emph{descent composition},
that is, the composition
of the lengths of its maximal weakly increasing blocks; a block ends
where $u_j>u_{j+1}$. Let $\mathcal W_d(\alpha)$ be the set of distinct
words obtained by inserting $d$ copies of the new largest letter
$\ell+1$ into $w(\alpha)$, keeping the old letters in order and
requiring the first letter to remain $\ell$.

\begin{theorem}[{D'Adderio et al.~\cite[Theorem~5.14]{DIIP26}}]\label{thmTheta}
For every composition $\alpha$ and every $d\geq0$,
\[
\Theta_{e_d}(D_\alpha\cdot1)
=
\sum_{u\in\mathcal W_d(\alpha)}
D_{\operatorname{Des}(u)}\cdot1.
\]
\end{theorem}

\begin{example}
Take $\alpha=(2,1)$ and $d=1$. Then $w(\alpha)=221$,
and the three permitted insertions are
\[
\begin{array}{c|c}
u&\operatorname{Des}(u)\\ \hline
2321&(2,1,1)\\
2231&(3,1)\\
2213&(2,2)
\end{array}
\]
Thus Theorem~\ref{thmTheta} gives
$\Theta_{e_1}(D_{(2,1)}\cdot1)
=(D_{(2,1,1)}+D_{(3,1)}+D_{(2,2)})\cdot1$.
For $k=3$ we have $\gamma^{(1)}=(2,1,0)$. By
Blasiak et al.~\cite[Lemma~3.6.2]{path}, a terminal zero may be omitted
when $D_\gamma$ acts on $1$; thus $D_{(2,1,0)}\cdot1=D_{(2,1)}\cdot1$.
\end{example}

\section{Row strict tableaux and
\texorpdfstring{$(km,kn)$}{(km,kn)}-preference functions}\label{sec:rst}
To compare parking and preference functions with the same
touch statistic, we use row strict tableaux (RSTs). We extend
the construction for Dyck paths by Blasiak et al.~\cite{Delta,path,dens}
to preference functions,
whose tableau rows may begin at negative positions.
Proposition~\ref{propdinvh} identifies the resulting tableau
statistics with area and dinv.

\subsection{The RST model}
We recall the tableau conventions of Blasiak et al.~\cite{Delta,path,dens}.

Let $\alpha\leq\beta$ be vectors in $\mathbb Z^{km}$. The shape $\beta/\alpha$ is a tuple of one-row skew shapes. In row $i$, the right edges of the boxes have integer coordinates
$\alpha_i+1,\alpha_i+2,\ldots,\beta_i$.
We call these the \emph{real boxes}. The boxes immediately to the left and right have right edges at coordinates $\alpha_i$ and $\beta_i+1$, respectively, and are called \emph{virtual boxes}. If $\alpha_i=\beta_i$, the row has no real boxes; its two virtual boxes are regarded as adjacent.

For $\sigma\in S_{km}$, a \emph{$\sigma$-triple} is a triple $(a,b,c)$ in which $a$ and $c$ are adjacent boxes in row $j$, while $b$ is a real box in row $i<j$. The boxes $a$ and $c$ may be virtual. If $\sigma(i)<\sigma(j)$, then $b$ has the same right-edge coordinate as $c$; if $\sigma(i)>\sigma(j)$, then $b$ has the same right-edge coordinate as $a$. We write $h_\sigma(\beta/\alpha)$ for the total number of such triples.

For a totally ordered alphabet $\mathcal A$, let
$\SSYT_{-}(\beta/\alpha,\mathcal A)$ be the set of RSTs:
fillings of the real boxes by elements of $\mathcal A$ that
increase strictly from left to right in every row. We set
$\SSYT_{-}(\beta/\alpha,\mathcal A)=\varnothing$ if
$\alpha\not\leq\beta$. An individual filling is denoted by $S$.
A $\sigma$-triple $(a,b,c)$ is \emph{increasing} in $S$ if
\[
  S(a)<S(b)<S(c),
\]
where the left and right virtual boxes have entries $-\infty$ and $+\infty$, respectively. Let $h_\sigma(S)$ be the number of increasing $\sigma$-triples in $S$.
For $S\in\SSYT_{-}(\beta/\alpha,\mathbb Z_{>0})$, we use
the weight notation $x^{\wt(S)} :=\prod_{u\in\beta/\alpha}x_{S(u)}$.

\begin{proposition}[{Blasiak et al.~\cite[Proposition~4.5.2]{path}}]\label{propNsym}
For $\alpha\leq\beta$ in $\mathbb Z^{km}$ and $\sigma\in S_{km}$, the generating function
\[
  N^\sigma_{\beta/\alpha}(X;q)
  :=\sum_{S\in\SSYT_{-}(\beta/\alpha,\mathbb Z_{>0})}
       q^{h_\sigma(S)}x^{\wt(S)}
\]
is symmetric in $X$.
\end{proposition}

\subsection{Connection with
\texorpdfstring{$(km,kn)$}{(km,kn)}-preference functions}
Each vertical run of a preference function gives one tableau
row. We shift the row by $c_j$, which is the height of the lowest
lattice point on or above the main diagonal in its column.
\begin{definition}\label{defalpha}
Let $k,m,n$ be positive integers with $\gcd(m,n)=1$. For $j\geq 0$, set
$c_j:=\left\lceil\frac{nj}{m}\right\rceil$.
For $0\leq x<km$, set $d_{x+1}:=c_x-\frac{nx}{m}+\varepsilon x$, where
$0<\varepsilon<1/(km^2)$. The numbers $d_1,\ldots,d_{km}$ are distinct,
and their relative order does not depend on $\varepsilon$.
Let $\sigma\in S_{km}$ be the permutation such that $\sigma(i)$ is the
position of $d_{km-i+1}$ when $d_1,\ldots,d_{km}$ are listed in
increasing order.

Let $\pi$ be a $(km,kn)$-preference function. For $1\leq i\leq km$, define
integers $\alpha_i\leq\beta_i$ by requiring that the vertical run of $\pi$ in
column $x=km-i$ go from height $c_{km-i}+\alpha_i$ to height
$c_{km-i}+\beta_i$; an empty run has $\alpha_i=\beta_i$.
Place the labels of the North steps in this column, from bottom to top,
in the real boxes of row $i$, from left to right. This gives an
RST $S\in\SSYT_{-}(\beta/\alpha,\mathbb Z_{>0})$.
\end{definition}

\begin{remark}
The value $\alpha_i$ is the signed vertical distance from the lowest
lattice point on or above the main diagonal in column $x=km-i$ to the
lower endpoint of the vertical run. Likewise, $\beta_i$ is the signed
distance to its upper endpoint.
\end{remark}

\begin{remark}
Unlike in the Dyck path case treated by
Blasiak et al.~\cite{Delta,path,dens}, both $\alpha_i$ and
$\beta_i$ can be negative.
Since $\pi$ starts at the origin and $c_0=0$, we have $\alpha_{km}=0$.
\end{remark}

\begin{proposition}\label{propdinvh}
Let $k,m,n$ be positive integers with $\gcd(m,n)=1$, and let $\pi$ be a
$(km,kn)$-preference function. With $\alpha,\beta,\sigma$, and $S$ as
in Definition~\ref{defalpha}, set
\[
r_*:=
\min_{1\leq i\leq km}
\bigl\{m(c_{km-i}+\alpha_i)-n(km-i)\bigr\}.
\]
Then
\[
\dinv(\pi)=h_\sigma(S), \qquad \area(\pi)=\sum_{i=1}^{km}\alpha_i-kr_*.
\]
\end{proposition}

\begin{proof}
We first prove the equality for $\dinv(\pi)$. Adjoin $kn$ North steps before
$\pi$ and $km$ East steps after it to obtain the $(2km,2kn)$-Dyck path $\pi^+:=\mathrm N^{kn}\,\pi\,\mathrm E^{km}$.
Let $\alpha^+,\beta^+,\sigma^+$ be its RST data from
Definition~\ref{defalpha}. The Ferrers diagrams used for
$\pdinv(\pi^+)$ and $\pdinv(\pi)$ coincide, so these statistics
are equal. By Blasiak et al.~\cite[Proposition~5.4.4]{path},
\[
h_{\sigma^+}(\beta^+/\alpha^+)
=\pdinv(\pi^+)=\pdinv(\pi).
\]

We claim that
\[
h_{\sigma^+}(\beta^+/\alpha^+)
=h_\sigma(\beta/\alpha)+B(\pi).
\]
The original boxes are shifted $kn$ columns to the right, while
the relative order of the numbers $d_{x+1}$ attached to their rows
is unchanged. Thus
triples involving only original boxes are unchanged. The appended
East steps give empty rows below them and create no triples.

Only the $kn$ boxes added to the row from column $x=0$ can contribute new
triples. This row has index $2km$, and $d_1=0<d_j$ for $j>1$;
hence $\sigma^+(2km)=1$. Every original real box $b$ in a new
triple is therefore aligned with the left member $a$ of an
adjacent pair in this row. The right edges of the new choices for
$a$ have coordinates $0,1,\ldots,kn-1$; the boundary at $kn$ corresponds
to the original left boundary and has already been counted.

An original North step from $(j,y)$ to $(j,y+1)$ gives a shifted
real box whose right edge has coordinate $kn+y+1-c_j$. Since $y\geq0$
and $c_j\leq kn$, this coordinate is positive. It therefore
aligns with one of the new choices for $a$ exactly when
\[
kn+y+1-c_j\leq kn-1
\quad\Longleftrightarrow\quad
y+1<c_j
\quad\Longleftrightarrow\quad
m(y+1)-nj<0.
\]
The last equivalence uses $c_j=\lceil nj/m\rceil$ and the fact
that $y+1$ is an integer. The last inequality says that the upper
endpoint $(j,y+1)$ of the North step lies strictly below the main
diagonal. Thus each step counted by $B(\pi)$
produces exactly one new triple, and there are no other new triples. This proves the claim and hence
\[
h_\sigma(\beta/\alpha)=\pdinv(\pi)-B(\pi).
\]

We next account for the labels. In every $\sigma$-triple $(a,b,c)$,
row strictness and the conventions for virtual boxes give $S(a)<S(c)$.
The triple fails to be increasing precisely when
$S(a)\geq S(b)$ or $S(b)\geq S(c)$. These two cases cannot occur
together, and a virtual box cannot give a failed inequality.

We claim that the attacks of $\pi$ correspond bijectively to the pairs
$(a,b)$ and $(b,c)$ of real boxes in $\sigma$-triples $(a,b,c)$, the step
of $b$ being the later step of the attack. For a real box $u$, write
$\operatorname{rank}(u)$ for the rank of its North step. The real box with
right-edge coordinate $p$ in row $km-x$ comes from the North step that
begins at $(x,c_x+p-1)$, so its rank is $m(p-1)+mc_x-nx$. Let $u$ be a
real box in row $j$ and $v$ a real box in row $i<j$, with right-edge
coordinates $p_u$ and $p_v$. The step of $v$ lies in a column to the right
of the step of $u$, so it comes later along $\pi$, and
\[
\operatorname{rank}(v)-\operatorname{rank}(u)=m(p_v-p_u+\theta),
\qquad
\theta:=\Bigl(c_{km-i}-\frac{n(km-i)}{m}\Bigr)
-\Bigl(c_{km-j}-\frac{n(km-j)}{m}\Bigr).
\]
Here $-1<\theta<1$ and $\theta\in\frac1m\mathbb Z$. Since
$d_{km-i+1}-d_{km-j+1}=\theta+\varepsilon(j-i)$ and
$0<\varepsilon(j-i)<1/m$, Definition~\ref{defalpha} gives
$\sigma(i)>\sigma(j)$ exactly when $\theta\geq0$. Hence
$0\leq\operatorname{rank}(v)-\operatorname{rank}(u)<m$ holds exactly when
either $p_v=p_u$ and $\sigma(i)>\sigma(j)$, or $p_v=p_u+1$ and
$\sigma(i)<\sigma(j)$; by the definition of $\sigma$-triples, this says
that $(u,v)$ is the pair $(a,b)$ of a $\sigma$-triple. In the same way,
$0<\operatorname{rank}(u)-\operatorname{rank}(v)\leq m$ holds exactly when
$(v,u)$ is the pair $(b,c)$ of a $\sigma$-triple. Since an area word entry
is the rank divided by $n$, these are the two kinds of attacks in the
definition of $\maxtdinv(\pi)$. Two North steps in the same column never
form an attack, because the later one has a rank that is larger by a
positive multiple of $m$. This proves the claim. Under this bijection, an
attack has increasing labels exactly when $S(a)<S(b)$, respectively
$S(b)<S(c)$. Thus $\maxtdinv(\pi)$ counts
all these pairs, while $\tdinv(\pi)$ counts those with strictly
increasing labels. Consequently,
\[
\maxtdinv(\pi)-\tdinv(\pi)
=h_\sigma(\beta/\alpha)-h_\sigma(S).
\]
Combining this equality with $h_\sigma(\beta/\alpha)=\pdinv(\pi)-B(\pi)$ gives
\begin{align*}
h_\sigma(S)
&=h_\sigma(\beta/\alpha)
  -\bigl(\maxtdinv(\pi)-\tdinv(\pi)\bigr)\\
&=\pdinv(\pi)-B(\pi)-\maxtdinv(\pi)+\tdinv(\pi)\\
&=\dinv(\pi).
\end{align*}
It remains to prove the equality for $\area(\pi)$. Recall that a North step beginning at
$(j,y)$ has rank $my-nj$. By the definition of area, its contribution
relative to the lowest rank $r_*$ is
\[
\left\lfloor\frac{my-nj-r_*}{n}\right\rfloor
=
\left\lfloor\frac{my-r_*}{n}\right\rfloor-j.
\]
The starting heights of the North steps are
$0,1,\ldots,kn-1$, each occurring once, while column $x=km-i$
contains $\beta_i-\alpha_i$ North steps. Therefore
\[
\area(\pi)
=
\sum_{y=0}^{kn-1}
\left\lfloor\frac{my-r_*}{n}\right\rfloor
-
\sum_{i=1}^{km}(km-i)(\beta_i-\alpha_i).
\]

Because $\gcd(m,n)=1$, the residues of $my$ modulo $n$, as $y$
ranges from $0$ to $kn-1$, contain each residue $k$ times. It follows
that
\[
\sum_{y=0}^{kn-1}
\left(
\left\lfloor\frac{my-r_*}{n}\right\rfloor
-
\left\lfloor\frac{my}{n}\right\rfloor
\right)
=
k\sum_{r=0}^{n-1}
\left\lfloor\frac{r-r_*}{n}\right\rfloor
=-kr_*.
\]
For the final equality, increasing $r_*$ by one decreases exactly one
of the $n$ terms by one, and the sum is zero when $r_*=0$.

For each $y$, the number $\lfloor my/n\rfloor$ counts the integers
$1\leq j<km$ satisfying $nj\leq my$. For a fixed $j$, this condition
holds for exactly $kn-c_j$ values of $y$ in
$\{0,\ldots,kn-1\}$. Hence
\[
\sum_{y=0}^{kn-1}\left\lfloor\frac{my}{n}\right\rfloor
=
\sum_{j=1}^{km-1}(kn-c_j)
=
\sum_{i=1}^{km-1}(kn-c_{km-i}).
\]
Let $y_i:=c_{km-i}+\alpha_i$ and $y_i':=c_{km-i}+\beta_i$ be the heights
of the lower and upper endpoints of the run in column $km-i$. The East
step after this run ends at the lower endpoint of the run in column
$km-i+1$, so $y_i'=y_{i-1}$ for $2\leq i\leq km$, and $y_1'=kn$. Thus
\begin{align*}
\sum_{i=1}^{km}(km-i)(\beta_i-\alpha_i)
&=\sum_{i=1}^{km}(km-i)(y_i'-y_i)\\
&=(km-1)kn-\sum_{i=1}^{km-1}y_i\\
&=(km-1)kn
  -\sum_{i=1}^{km-1}(c_{km-i}+\alpha_i).
\end{align*}
Substituting these identities into the area formula yields
\[
\area(\pi)
=
\sum_{i=1}^{km-1}\alpha_i-kr_*
=
\sum_{i=1}^{km}\alpha_i-kr_*,
\]
where the last equality uses $\alpha_{km}=0$.
\end{proof}

Together with Proposition~\ref{propNsym}, Proposition~\ref{propdinvh}
shows that, for a fixed path, the sum of $q^{\dinv(\pi)}x_\pi$ over all
labelings of the path is the symmetric function
$N^\sigma_{\beta/\alpha}(X;q)$.

\begin{example}\label{example68}
Continue with the path $\pi$ of Example~\ref{examplePathStats}.
Figure~\ref{figPathTableau} shows its associated RST $S$ beside
the path. The column $x=0$ gives the top row $i=6$, and
$x=5$ gives the bottom row $i=1$. The resulting data are
\[
\alpha=(0,-2,-1,0,0,0),
\beta=(1,1,0,0,1,2),
(\sigma(1),\ldots,\sigma(6))=(4,6,2,3,5,1).
\]

\begin{center}
\begin{minipage}{\linewidth}
\centering
\begin{tikzpicture}[x=5.5mm,y=5.5mm]
  \draw[step=1,gray!35,very thin] (0,0) grid (6,8);
  \draw[blue,dashed,line width=.85pt] (0,0)--(6,8);
  \draw[blue!65,densely dashed,line width=.7pt] (1,0)--(6,6.6667);
  \draw[line width=1pt]
    (0,0)--(0,2)--(1,2)--(1,3)--(2,3)--(3,3)
    --(3,4)--(4,4)--(4,7)--(5,7)--(5,8)--(6,8);
  \foreach \x/\y/\z in
    {0/.5/2,0/1.5/6,1/2.5/3,3/3.5/1,
     4/4.5/4,4/5.5/5,4/6.5/7,5/7.5/8}
    \node[font=\scriptsize,fill=white,inner sep=.6pt] at ({\x+.5},\y) {$\z$};
  \node[font=\small] at (3,-.7) {$\pi$};

  \draw[blue!55,densely dashed,line width=.4pt]
    (11,-.35)--(11,8.35);
  \node[blue!70,font=\scriptsize,anchor=south]
    at (11,8.43) {$x=0$};
  \def\RBox#1#2#3{%
    \draw ({10+#2},#1) rectangle ({11+#2},{#1+1});%
    \node[font=\scriptsize] at ({10.5+#2},{#1+.5}) {$#3$};}
  \def\VLeft#1#2{%
    \node[font=\tiny,anchor=center,inner sep=0pt] at ({10.5+#2},{#1+.5}) {$-\infty$};}
  \def\VRight#1#2{%
    \node[font=\tiny,anchor=center,inner sep=0pt] at ({10.5+#2},{#1+.5}) {$+\infty$};}

  \VLeft{7}{0}\RBox{7}{1}{2}
  \RBox{7}{2}{6}\VRight{7}{3}

  \VLeft{5.6}{0}\RBox{5.6}{1}{3}
  \VRight{5.6}{2}

  \VLeft{4.2}{0}\VRight{4.2}{1}
  \draw (11,4.2)--(11,5.2);

  \VLeft{2.8}{-1}\RBox{2.8}{0}{1}
  \VRight{2.8}{1}

  \VLeft{1.4}{-2}\RBox{1.4}{-1}{4}
  \RBox{1.4}{0}{5}\RBox{1.4}{1}{7}
  \VRight{1.4}{2}

  \VLeft{0}{0}\RBox{0}{1}{8}
  \VRight{0}{2}

  \node[font=\small] at (10.5,-.7) {$S$};
\end{tikzpicture}

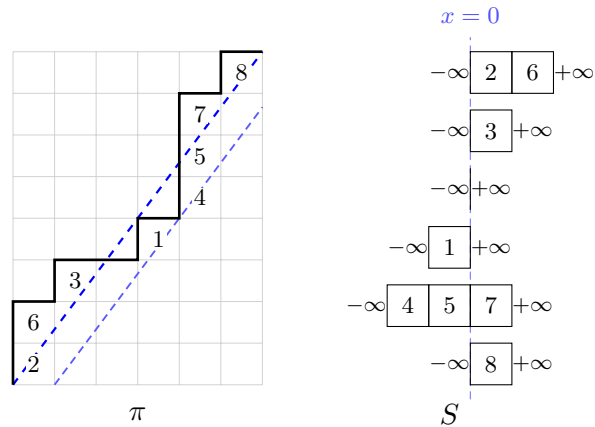
\captionof{figure}{Example of a $(6,8)$-preference function and its associated RST}\label{figPathTableau}
\end{minipage}
\end{center}

The run in column $x=4$ goes from $y=4$ to $y=7$.
Since $\lceil 4n/m\rceil=\lceil 16/3\rceil=6$, it gives row $i=2$
of $S$, whose boxes have right edge coordinates $-1,0,1$
and entries $4,5,7$. This explains $\alpha_2=-2$
and $\beta_2=1$. The empty run in column $x=2$ gives
the empty row $i=4$, with $\alpha_4=\beta_4=0$.

The minimum North step rank is $r_*=-4$. In tableau coordinates,
Proposition~\ref{propdinvh} gives
\[
\area(\pi)
=(0-2-1+0+0+0)-2(-4)
=5.
\]
This agrees with the area calculation in Example~\ref{examplePathStats}.

The same example gives $\pdinv(\pi)=13$, $B(\pi)=1$,
$\maxtdinv(\pi)=15$, and $\tdinv(\pi)=8$.
The proof of Proposition~\ref{propdinvh} therefore gives
$13-1=12$ $\sigma$-triples of $\beta/\alpha$.
Exactly $15-8=7$ of them fail the increasing label condition,
leaving $5$ increasing triples in $S$. The triple with entries
$(2,3,6)$ is among these $5$. Consequently,
\[
\dinv(\pi)=13-1-15+8=5=h_\sigma(S).
\]
\end{example}

\section{Main proof}\label{sec:main}
We first identify $E_{k,\ell}[-MX^{m,n}]\cdot1$ with a sum over
parking functions of touch $\ell$ (Theorem~\ref{thmEnkPF}), using Mellit's
compositional shuffle theorem (Theorem~\ref{kmknshuffle}).
We then encode the row strict tableaux of Section~\ref{sec:rst} by words in
$U,V,E$. Proposition~\ref{propqcyc} compares words with the
same profile but different initial states and gives the
sum over preference functions in Theorem~\ref{thmEnkPr}.
Summing over $\ell$ proves Theorem~\ref{thmmain}.
The final subsection applies the combinatorial skewing argument to this
result and then uses the Negu\textcommabelow{t} element identities of Section~\ref{subsec:negut}
to prove Theorem~\ref{thmDelta}.

\subsection{A touch refinement}
As in Sergel's proof~\cite{Ser17}, we first evaluate
$E_{k,\ell}[-MX^{m,n}]\cdot1$ using parking functions.
\begin{theorem}\label{thmEnkPF}
For $1\le\ell\le k$,
\begin{equation}\label{eqEnkPF}
E_{k,\ell}[-MX^{m,n}]\cdot 1
=
\sum_{\substack{\pi\in\PF_{km,kn}\\
                 \touch(\pi)=\ell}}
q^{\dinv(\pi)}t^{\area(\pi)}x_\pi.
\end{equation}
\end{theorem}

\begin{proof}
By Proposition~\ref{propEtoC}, $E_{k,\ell}$ is the sum of
$C_\alpha[X;q]$ over all $\alpha\vDash k$ with $\ell(\alpha)=\ell$.
Apply the linear map $f\mapsto f[-MX^{m,n}]\cdot1$ and
Theorem~\ref{kmknshuffle} to each term. Every $\pi\in\PF_{km,kn}$ has
exactly one touch composition $\comp(\pi)\vDash k$, and
$\touch(\pi)$ is its number of parts (Section~\ref{subsec:comb}).
This gives~\eqref{eqEnkPF}.
\end{proof}

The preference function counterpart extends Sergel's touch
refinement~\cite[Theorem~4.3]{Ser17} to $(km,kn)$-paths.
\begin{theorem}\label{thmEnkPr}
For $1\le\ell\le k$,
\begin{equation}\label{eqEnkPr}
\frac{[km]_q}{[\ell]_q}
E_{k,\ell}[-MX^{m,n}]\cdot 1
=
\sum_{\substack{\pi\in\Pref_{km,kn}\\
                 \touch(\pi)=\ell}}
q^{\dinv(\pi)}t^{\area(\pi)}x_\pi.
\end{equation}
\end{theorem}

By Theorem~\ref{thmEnkPF}, Theorem~\ref{thmEnkPr} is
equivalent to
\begin{equation}\label{eqmaincomb}
[\ell]_q
\sum_{\substack{\pi\in\Pref_{km,kn}\\
                 \touch(\pi)=\ell}}
q^{\dinv(\pi)}t^{\area(\pi)}x_\pi
=
[km]_q
\sum_{\substack{\pi\in\PF_{km,kn}\\
                 \touch(\pi)=\ell}}
q^{\dinv(\pi)}t^{\area(\pi)}x_\pi.
\end{equation}

We now deduce the $(km,kn)$-square paths theorem from
Theorem~\ref{thmEnkPr}, which is proved in Section~\ref{subsec:words}.
\begin{theorem}\label{thmpn}
For all positive integers $k,m,n$ with $\gcd(m,n)=1$,
\begin{equation}\label{eqpn}
(-1)^{k-1}\frac{[km]_q}{[k]_q}
p_k[-MX^{m,n}]\cdot 1
=
\sum_{\pi\in\Pref_{km,kn}}
q^{\dinv(\pi)}t^{\area(\pi)}x_\pi.
\end{equation}
\end{theorem}

\begin{proof}
Apply the linear map
$f\mapsto f[-MX^{m,n}]\cdot1$ to Proposition~\ref{propptoEnk}:
\[
(-1)^{k-1}p_k[-MX^{m,n}]\cdot1
=
\sum_{\ell=1}^k
\frac{[k]_q}{[\ell]_q}
E_{k,\ell}[-MX^{m,n}]\cdot1.
\]
Multiplication by $[km]_q/[k]_q$ changes the coefficient
of the term indexed by $\ell$ to $[km]_q/[\ell]_q$.
Theorem~\ref{thmEnkPr} identifies that term with the
weighted sum over preference functions of touch $\ell$.
Thus
\[
(-1)^{k-1}\frac{[km]_q}{[k]_q}
p_k[-MX^{m,n}]\cdot1
=
\sum_{\ell=1}^k
\sum_{\substack{\pi\in\Pref_{km,kn}\\\touch(\pi)=\ell}}
q^{\dinv(\pi)}t^{\area(\pi)}x_\pi.
\]
Every preference function $\pi$ satisfies
$1\leq\touch(\pi)\leq k$ by Section~\ref{subsec:comb}. Hence the inner
sums partition $\Pref_{km,kn}$, giving the right-hand side
of~\eqref{eqpn}.
\end{proof}

\subsection{Encoding row strict tableaux by words}\label{subsec:words}
Fix the finite alphabet $\{1,\ldots,L\}$; that is, we work with the
variables $x_1,\ldots,x_L$. In row $km-x$, let $y$ denote the
coordinate of the left edge of a real box or an integer boundary
between boxes, and define the rank of that boundary by
\[
\rho(x,y):=m(c_x+y)-nx.
\]
This formula applies to the boundary of an empty row and to the
terminal boundary of a nonempty row. We do not assign ranks to
virtual boxes. Set
\[
\rho_*:=\min_{0\leq x<km}\rho(x,\alpha_{km-x})=r_*.
\]
The minimum on the left agrees with the minimum rank $r_*$ of a North
step in the path. Indeed, along a nonempty row ranks increase by $m$
at each real box, so its first North step has the smallest rank in
that row. An empty row followed by another row has a successor
boundary of rank $n$ less. Each boundary in a terminal string of
empty rows has rank at least $n>0$, whereas the first North step has
rank at most $0$. Thus an empty row cannot attain the minimum.

A \emph{state} is a pair $(r,i)$, where $r=\rho(x,y)-\rho_*$ is the
\emph{level} of the current boundary and $i$ is
the label currently inspected. The value $i=L+1$ marks the end
of a row. The letters $U_{r,i}$, $V_{r,i}$, and $E_r$ record
transitions between states, and a sequence of these letters is a
\emph{word}. A word is a \emph{walk} if each letter departs from the
state at which the previous letter arrives, and a walk is \emph{closed}
if its last letter arrives at the state from which its first letter
departs. We call a word \emph{proper} if it is obtained by
scanning the RST associated with a $(km,kn)$-preference function as
in Definition~\ref{defalpha}.

Scan the rows from top to bottom, so that $x$ runs from $0$ to
$km-1$. In row $km-x$, start at $y=\alpha_{km-x}$ and inspect
$i=1,\ldots,L$ in order. The boundary moves one unit to the right
when the entry $i$ occurs. At state $(r,i)$, record the following
letter and transition:
\begin{enumerate}
    \item $U_{r,i}$ records that the real box immediately to the
    right of the current boundary contains $i$. Move the boundary
    one unit to the right. This letter makes the transition
    \[
    U_{r,i}:=\bigl((r,i)\longrightarrow(r+m,i+1)\bigr).
    \]

    \item $V_{r,i}$ records that the real box immediately to the
    right of the current boundary does not contain $i$, or that no
    real box lies to its right. Keep the boundary fixed. This
    letter makes the transition
    \[
    V_{r,i}:=\bigl((r,i)\longrightarrow(r,i+1)\bigr).
    \]

    \item $E_r$ records that all labels $1,\ldots,L$ in the row
    have been inspected. The current state is $(r,L+1)$, at the
    boundary $y=\beta_{km-x}$. This letter makes the transition
    \[
    E_r:=\bigl((r,L+1)\longrightarrow(r-n,1)\bigr).
    \]
\end{enumerate}
The same rules apply to empty rows. We will check below that the
final $E$-letter returns to the initial state, so the word is closed.

Moving the boundary one unit to the right increases its rank by $m$.
The East step after the run in column $x$ ends where the run in
column $x+1$ begins, so
\[
\rho(x+1,\alpha_{km-x-1})
=
\rho(x,\beta_{km-x})-n
\qquad(0\leq x<km-1).
\]
The final $E$-transition also decreases the rank by $n$ and returns
to the first row, since
\[
\rho(km-1,\beta_1)
=mkn-n(km-1)=n,
\qquad
\rho(0,\alpha_{km})=0.
\]
Each row contributes one $U$- or $V$-letter for every
$1\leq i\leq L$, followed by an $E$-letter. Thus a proper word is a
closed walk. Its $U$-letters give the entries of each row in
scan order, so the word recovers the RST.

The \emph{profile} $\mathfrak p$ records the number of occurrences of
each letter. For a fixed profile, write
\[
u_{r,i}:=\#\{U_{r,i}\},\qquad
v_{r,i}:=\#\{V_{r,i}\},\qquad
w_r:=\#\{E_r\},\qquad
g_{r,i}:=u_{r,i}+v_{r,i}
\quad(1\leq i\leq L).
\]
Thus $g_{r,i}$ is the number of letters departing from the state $(r,i)$.
Set $g_{r,L+1}:=w_r$, and interpret counts at negative levels as zero.
A profile is \emph{realizable}
if it is the profile of at least one proper word.

We will also use the converse: every closed walk with a
realizable profile that begins at $(a,1)$ is a proper word.
Between consecutive $E$-letters, the second coordinate runs through
$1,\ldots,L+1$. The $U$-letters in that segment therefore specify a
row with strictly increasing entries. Since the profile is realizable,
the walk has $km$ $E$-letters and $kn$ $U$-letters. Its successive
rows give a $(km,kn)$-path ending with an East step and a valid
labeling.

It remains to check the levels. After $y$ $U$-letters and $x$
$E$-letters, the level is $a+my-nx$. This is the level 
$my-nx-r_*$ of the corresponding boundary when $r_*=-a$.
All levels in the walk are nonnegative: the same transitions occur
in a proper word realizing the profile. That word has a row beginning
at level $0$, by the minimum rank argument above. The new walk also
visits this row start, because it uses the same letters and every
arrival at a state is paired with a departure. A row starting at
level $0$ cannot be empty, since its $E$-transition would arrive at
level $-n$. Thus the reconstructed path has minimum North-step rank
$r_*$. Its boundaries and labels agree with those of its associated
RST, proving the converse. In particular, rotating a word at a row
start, or constructing a closed walk from a realizable profile, still
gives a proper word.

\begin{lemma}\label{lemmabalance}
For $1\leq i\leq L$,
\[
g_{r,i+1}=v_{r,i}+u_{r-m,i}
=g_{r,i}-u_{r,i}+u_{r-m,i}.
\]
\end{lemma}

\begin{proof}
Every arrival at state $(r,i+1)$ is made either by a $V_{r,i}$ from
level $r$ or by a $U_{r-m,i}$ from level $r-m$. These possibilities
are disjoint, and counting them gives the result.
\end{proof}

\begin{remark}
For $i=L$, Lemma~\ref{lemmabalance} includes
$w_r=v_{r,L}+u_{r-m,L}$. Likewise, the letters arriving at
$(r,1)$ are precisely the $E_{r+n}$ letters. Hence
\begin{equation}\label{eqbal2}
g_{r,1}=w_{r+n}=g_{r+n,L+1}.
\end{equation}
\end{remark}

Let $\nu_r$ be the number of rows whose initial boundary has
level $r$. The transition between rows, including the
transition from the last row to the first, gives
\[
\nu_r:=g_{r,1}=w_{r+n}.
\]
Each of the $km$ rows has one initial state, so
\[
\sum_{r\geq0}\nu_r=km.
\]
By definition, $\nu_0$ counts rows whose initial boundary has level
$0$. Such a row is nonempty: otherwise its $E$-transition would
arrive at level $-n$. Under the RST correspondence, its initial
boundary is the starting point of the first North step in that
column. Level $0$ means that this step has rank $r_*$, so these
rows correspond exactly to the North steps counted by $\touch(\pi)$. Hence
\[
\nu_0=\touch(\pi).
\]

\begin{example}\label{exampleRowWords}
Continue with the path of Example~\ref{examplePathStats}.
Figure~\ref{figRowWords} displays the rows of its RST beside the
corresponding parts of its word.
\begin{center}
\begin{minipage}{\linewidth}
\centering
\begin{tikzpicture}[x=5.5mm,y=5.5mm]
  \draw[blue!55,densely dashed,line width=.4pt]
    (2,-.35)--(2,8.35);
  \node[blue!70,font=\scriptsize,anchor=south]
    at (2,8.43) {$x=0$};
  \def\RBox#1#2#3{%
    \draw ({1+#2},#1) rectangle ({2+#2},{#1+1});%
    \node[font=\scriptsize] at ({1.5+#2},{#1+.5}) {$#3$};}
  \def\VLeft#1#2{%
    \node[font=\tiny,anchor=center,inner sep=0pt] at ({1.5+#2},{#1+.5}) {$-\infty$};}
  \def\VRight#1#2{%
    \node[font=\tiny,anchor=center,inner sep=0pt] at ({1.5+#2},{#1+.5}) {$+\infty$};}
  \VLeft{7}{0}\RBox{7}{1}{2}
  \RBox{7}{2}{6}\VRight{7}{3}
  \VLeft{5.6}{0}\RBox{5.6}{1}{3}
  \VRight{5.6}{2}
  \VLeft{4.2}{0}\VRight{4.2}{1}
  \draw (2,4.2)--(2,5.2);
  \VLeft{2.8}{-1}\RBox{2.8}{0}{1}
  \VRight{2.8}{1}
  \VLeft{1.4}{-2}\RBox{1.4}{-1}{4}
  \RBox{1.4}{0}{5}\RBox{1.4}{1}{7}
  \VRight{1.4}{2}
  \VLeft{0}{0}\RBox{0}{1}{8}
  \VRight{0}{2}
  \node[anchor=center,align=center,font=\scriptsize,inner sep=0pt] at (12,7.5)
    {$\begin{gathered}V_{4,1}\,U_{4,2}\,V_{7,3}\,V_{7,4}\,V_{7,5}\\[-2pt]
     U_{7,6}\,V_{10,7}\,V_{10,8}\,E_{10}\end{gathered}$};
  \node[anchor=center,align=center,font=\scriptsize,inner sep=0pt] at (12,6.1)
    {$\begin{gathered}V_{6,1}\,V_{6,2}\,U_{6,3}\,V_{9,4}\,V_{9,5}\\[-2pt]
     V_{9,6}\,V_{9,7}\,V_{9,8}\,E_9\end{gathered}$};
  \node[anchor=center,align=center,font=\scriptsize,inner sep=0pt] at (12,4.7)
    {$\begin{gathered}V_{5,1}\,V_{5,2}\,V_{5,3}\,V_{5,4}\,V_{5,5}\\[-2pt]
     V_{5,6}\,V_{5,7}\,V_{5,8}\,E_5\end{gathered}$};
  \node[anchor=center,align=center,font=\scriptsize,inner sep=0pt] at (12,3.3)
    {$\begin{gathered}U_{1,1}\,V_{4,2}\,V_{4,3}\,V_{4,4}\,V_{4,5}\\[-2pt]
     V_{4,6}\,V_{4,7}\,V_{4,8}\,E_4\end{gathered}$};
  \node[anchor=center,align=center,font=\scriptsize,inner sep=0pt] at (12,1.9)
    {$\begin{gathered}V_{0,1}\,V_{0,2}\,V_{0,3}\,U_{0,4}\,U_{3,5}\\[-2pt]
     V_{6,6}\,U_{6,7}\,V_{9,8}\,E_9\end{gathered}$};
  \node[anchor=center,align=center,font=\scriptsize,inner sep=0pt] at (12,.5)
    {$\begin{gathered}V_{5,1}\,V_{5,2}\,V_{5,3}\,V_{5,4}\,V_{5,5}\\[-2pt]
     V_{5,6}\,V_{5,7}\,U_{5,8}\,E_8\end{gathered}$};
  \node[font=\small] at (2.5,-.75) {RST $S$};
\node[font=\small] at (12,-.75) {word};
\end{tikzpicture}

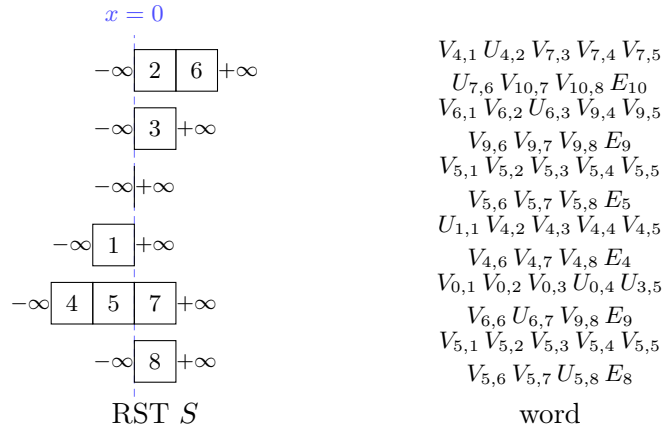
\captionof{figure}{An example of an RST and its word}\label{figRowWords}
\end{minipage}
\end{center}
From top to bottom, the rows correspond to the path columns
$x=0,1,\ldots,5$. Each line
ends with the $E$ that moves the scan to the next row.
The top row begins at level $4=\rho(0,0)-\rho_*$. Its entries
$2$ and $6$ give $U_{4,2}$ and $U_{7,6}$: each $U$
crosses a real box and raises the boundary level by $m=3$.
The other labels give $V$'s. The final $E_{10}$
lowers the level by $n=4$, so the next row begins at level
$6$, where its entry $3$ gives $U_{6,3}$. The empty third
row has only $V$'s before $E_5$. The first row's $V_{4,1}$
and the fourth row's $U_{1,1}$ are the two arrivals at
$(4,2)$, so $g_{4,2}=v_{4,1}+u_{1,1}=2$ by
Lemma~\ref{lemmabalance}.
The initial levels are $4,6,5,1,0,5$, so the profile has
\[
(\nu_0,\nu_1,\nu_2,\nu_3,\nu_4,\nu_5,\nu_6)=(1,1,0,0,1,2,1).
\]
In particular, $\nu_5=w_9=2$, $\nu_6=w_{10}=1$, and
$\nu_0=\touch(\pi)=1$.
\end{example}

For a proper word $w$, define
\[
h_\sigma(w)
:=
\#\bigl\{(V_{z,i},U_{r,i}):
V_{z,i}\text{ occurs before }U_{r,i},\
r<z\leq r+m\bigr\}.
\]

\begin{proposition}\label{propshift}
Let $\pi$ be a $(km,kn)$-preference function, $S$ its associated
RST, and $w$ its word. Then
\[
x^{\wt(S)}
=
\prod_{i=1}^{L}x_i^{\sum_r u_{r,i}},
\qquad
\area(\pi)
=
\sum_{r,i}u_{r,i}\left\lfloor\frac{r}{n}\right\rfloor,
\qquad
h_\sigma(w)=h_\sigma(S).
\]
\end{proposition}

\begin{proof}
An occurrence of $U_{r,i}$ records a real box containing $i$.
Counting these occurrences by label gives the asserted expression
for $x^{\wt(S)}$.

For $\area(\pi)$, consider the North step corresponding to an
occurrence of $U_{r,i}$. It begins at $(x,c_x+y)$, so its entry
in the area word is $\rho(x,y)/n$. Since $s=-r_*/n$, this step contributes
\[
\left\lfloor\frac{\rho(x,y)}{n}+s\right\rfloor
=
\left\lfloor\frac{\rho(x,y)-r_*}{n}\right\rfloor
=
\left\lfloor\frac{r}{n}\right\rfloor.
\]
Summing over all $U$-letters gives the asserted expression for
$\area(\pi)$.

It remains to compare the two $h_\sigma$ statistics. Consider an
occurrence of $V_{z,i}$ in the earlier row $km-x'$ and an occurrence
of $U_{r,i}$ in the later row $km-x$, where $x'<x$. Let $y'$ and
$y$ be the coordinates of their boundaries. Their levels
satisfy
\[
z-r
=
n(x-x')-m(c_x-c_{x'})+m(y'-y),
\qquad
-m<n(x-x')-m(c_x-c_{x'})<m.
\]
Moreover,
\[
d_{x+1}-d_{x'+1}
=
-\frac{n(x-x')-m(c_x-c_{x'})}{m}
+\varepsilon(x-x').
\]
Since $0<\varepsilon<1/(km^2)$ by Definition~\ref{defalpha}, we have
$0<m\varepsilon(x-x')<1$ for every $0\leq x'<x<km$. If the
integer $n(x-x')-m(c_x-c_{x'})$ is positive, then
$d_{x+1}<d_{x'+1}$. By the definition of $\sigma$, the later real
box must align with the right member of the adjacent pair in the
earlier row. This means $y'=y$. If the integer is nonpositive,
then $d_{x+1}>d_{x'+1}$; the positive term
$\varepsilon(x-x')$ resolves the equality case. The later box
must now align with the left member, which means $y'=y+1$.
Since the integer lies strictly between $-m$ and $m$, the condition
$0<z-r\leq m$ holds exactly in these two cases. Any other integer
value of $y'-y$ puts $z-r$ outside this interval.

The letter $V_{z,i}$ marks the adjacent pair of boxes in the earlier
row that brackets $i$: entries to its left are smaller than $i$ and
entries to its right are larger. This remains true at the ends of a
row by the conventions for virtual entries $-\infty$ and $+\infty$.
Thus a pair $(V_{z,i},U_{r,i})$ counted by $h_\sigma(w)$ determines
an increasing $\sigma$-triple of $S$. Conversely, in every increasing
$\sigma$-triple whose middle box contains $i$, the adjacent pair in
the earlier row brackets $i$. Its scan therefore records the unique
corresponding $V_{z,i}$, while the middle box records $U_{r,i}$.
This proves $h_\sigma(w)=h_\sigma(S)$.
\end{proof}

\begin{example}\label{exampleWordStats}
Continue with the same path $\pi$. Its RST $S$ is shown in
Figure~\ref{figPathTableau}, and its word $w$ in Figure~\ref{figRowWords}.

The $U$-levels in these six rows are
$(4,7),(6),\varnothing,(1),(0,3,6),(5)$.
By Proposition~\ref{propshift}, each $U_{r,i}$ contributes
$\lfloor r/4\rfloor$ to the area. The row contributions are
$2,1,0,0,1,1$, so $\area(\pi)=5$, as in
Example~\ref{example68}. The $U$-letters contain each label
$1,\ldots,8$ once. They therefore give
$x^{\wt(S)}=x_1x_2\cdots x_8=x_\pi$.

The pairs that contribute to $h_\sigma(w)$ are
\[
(V_{7,3},U_{6,3}),\quad
(V_{4,1},U_{1,1}),\quad
(V_{5,5},U_{3,5}),\quad
(V_{4,5},U_{3,5}),\quad
(V_{9,7},U_{6,7}).
\]
Each of them pairs a $V$ in an earlier row with a $U$ in a later row
and corresponds to an increasing $\sigma$-triple in $S$, as
proved in Proposition~\ref{propshift}. Thus
$h_\sigma(w)=5=h_\sigma(S)=\dinv(\pi)$.
\end{example}

Let $\mathfrak W_a(\mathfrak p)$ be the set of proper words with
profile $\mathfrak p$ and initial state $(a,1)$. Set
\[
\phi_a(\mathfrak p;q)
:=
\sum_{w\in\mathfrak W_a(\mathfrak p)}q^{h_\sigma(w)}.
\]

For $a\geq0$, set
\[
P_a:=\sum_{b<a}\nu_b.
\]

\begin{proposition}\label{propqcyc}
Let $\mathfrak p$ be a realizable profile with $\nu_0=\ell$.
For every $a\geq0$,
\[
[\ell]_q\phi_a(\mathfrak p;q)
=
q^{P_a}[\nu_a]_q\phi_0(\mathfrak p;q).
\]
\end{proposition}

We now deduce Theorem~\ref{thmEnkPr} from
Proposition~\ref{propqcyc}.

\begin{proof}[Proof of Theorem~\ref{thmEnkPr}]
Fix a realizable profile $\mathfrak p$ with $\nu_0=\ell$.
Only finitely many $\nu_a$ are nonzero, so summing
Proposition~\ref{propqcyc} over $a\geq0$ gives
\[
[\ell]_q\sum_{a\geq0}\phi_a(\mathfrak p;q)
=
\phi_0(\mathfrak p;q)
\sum_{a\geq0}q^{P_a}[\nu_a]_q.
\]
Since $P_0=0$ and $P_{a+1}=P_a+\nu_a$, the nonzero terms
\[
q^{P_a}[\nu_a]_q
=
q^{P_a}+q^{P_a+1}+\cdots+q^{P_a+\nu_a-1}
\]
occupy consecutive, disjoint sets of powers of $q$. They cover the
powers from $q^0$ through $q^{km-1}$ because
$\sum_{a\geq0}\nu_a=km$. Hence
\[
\sum_{a\geq0}q^{P_a}[\nu_a]_q=[km]_q.
\]
Consequently,
\begin{equation}\label{eqfixedprofile}
[\ell]_q\sum_{a\geq0}\phi_a(\mathfrak p;q)
=
[km]_q\phi_0(\mathfrak p;q).
\end{equation}

The first row of an RST associated with a path begins at absolute rank
$\rho(0,\alpha_{km})=0$, so its initial level is $a=-r_*$. The
converse construction above shows that, as $a$ varies, the sets
$\mathfrak W_a(\mathfrak p)$ account for exactly the preference
functions with profile $\mathfrak p$. The words in
$\mathfrak W_0(\mathfrak p)$ account for the parking functions.
Indeed, $a=0$ means $r_*=0$. Ranks increase along a vertical run,
so it is enough to check its lower endpoint. The run in column
$km-i$ begins at $(km-i,c_{km-i}+\alpha_i)$, so nonnegative rank means
$m(c_{km-i}+\alpha_i)-n(km-i)\geq0$. Since $c_{km-i}=\lceil n(km-i)/m\rceil$
and $\alpha_i$ is an integer, this is equivalent to $\alpha_i\geq0$. The
condition $\alpha_i\geq0$ for every $i$ says that the path stays weakly
above the main diagonal.

Proposition~\ref{propshift} shows that the area and label monomial
are constant on a fixed profile. It also gives
$h_\sigma(w)=h_\sigma(S)$, which equals $\dinv(\pi)$ by
Proposition~\ref{propdinvh}. Multiplying~\eqref{eqfixedprofile} by
the common factor $t^{\area(\pi)}x_\pi$ therefore gives
\eqref{eqmaincomb} for this profile. Summing over all realizable
profiles with $\nu_0=\ell$ proves~\eqref{eqmaincomb}.
Theorem~\ref{thmEnkPF} identifies its sum over parking functions with
$E_{k,\ell}[-MX^{m,n}]\cdot1$, giving~\eqref{eqEnkPr}. Finally,
$L$ was arbitrary, so the identity holds in the full alphabet $X$.
\end{proof}

\subsection{Proof of Proposition~\ref{propqcyc}}
Set
\[
I(w):=
\#\left\{
\bigl(U_{r,i},V_{r,i}\bigr)\text{ as a pair of occurrences in }w:
U_{r,i}\text{ occurs before }V_{r,i}
\right\}.
\]
Both letters in such a pair depart from the state $(r,i)$.

To compare $h_\sigma$ on words with the same profile, we first compute
how the statistic $J$ below changes when the first letter of a word is
moved to the end. This also controls the exchange $ACB\mapsto BCA$ in
Lemma~\ref{lemmaExchange}.

Set $J(w):=h_\sigma(w)-I(w)$. For ordered pairs of letter
occurrences, define
\[
\kappa(V_{z,i},U_{r,i})
:=\mathbf{1}_{r<z\leq r+m},
\qquad
\kappa(U_{r,i},V_{z,i})
:=-\mathbf{1}_{z=r},
\]
and set $\kappa(e,f):=0$ for all other pairs. The first
value counts the pairs in $h_\sigma$, while the second
subtracts the pairs in $I$. Thus, if
$w=e_1e_2\cdots e_N$, then
\[
J(w)=\sum_{1\leq j<j'\leq N}\kappa(e_j,e_{j'}).
\]
For a profile $\mathfrak p$, set
\[
\delta_{\mathfrak{p}}(e)
:=\sum_{f\text{ in }\mathfrak{p}}
\bigl(\kappa(f,e)-\kappa(e,f)\bigr).
\]
When $e_1e_2\cdots e_N$ has profile $\mathfrak p$, this is
$J(e_2\cdots e_Ne_1)-J(e_1e_2\cdots e_N)$: moving the first
letter to the end reverses its order with every other letter.
Also set
\[
G_{r,i}
:=
\sum_{s<r}g_{s,i}
-
\sum_{j\geq1}g_{r+jm,i}.
\]
\begin{lemma}\label{lemmaDeltaG}
Suppose a profile $\mathfrak p$ has equally many arrivals and
departures at every state. Then
\begin{align}
\delta_{\mathfrak p}(U_{r,i})&=\sum_{z=r}^{r+m}v_{z,i},\label{eqDeltaU}\\
\delta_{\mathfrak p}(V_{r,i})&=-\sum_{z=r-m}^{r}u_{z,i},\label{eqDeltaV}\\
\delta_{\mathfrak p}(E_r)&=0.\label{eqDeltaE}
\end{align}
For every letter $e$, its change is the difference between the
values of $G$ at its arrival and departure states:
\[
\delta_{\mathfrak p}(e)
=G_{\text{arrival state of }e}
-G_{\text{departure state of }e}.
\]
\end{lemma}

\begin{proof}
We compute this change for each letter type. When the first
letter is $U_{r,i}$, moving it past a $V_{r,i}$ removes
a contribution of $-1$, and moving it past a $V_{z,i}$
with $r<z\leq r+m$ creates a contribution of $1$.
No other pair changes $J$, giving \eqref{eqDeltaU}.

When the first letter is $V_{r,i}$, moving it past a
$U_{z,i}$ with $r-m\leq z<r$ removes a contribution
of $1$. At $z=r$, the reversed pair instead creates a
contribution of $-1$. Each such pair therefore changes
$J$ by $-1$; all other pairs do not affect it. This gives
\eqref{eqDeltaV}.
Pairs containing an $E$-letter contribute zero, so
\eqref{eqDeltaE} follows.

We now prove the last equality in Lemma~\ref{lemmaDeltaG}. The balance equations
in Lemma~\ref{lemmabalance} and \eqref{eqbal2} hold for the
profile in the statement, since each state has equally many
arrivals and departures.

A letter $U_{r,i}$ goes from $(r,i)$ to
$(r+m,i+1)$. By Lemma~\ref{lemmabalance},
\begin{align*}
G_{r+m,i+1}-G_{r,i}
&=
\left(
\sum_{s<r+m}g_{s,i+1}
-\sum_{s<r}g_{s,i}
\right)
+\sum_{j\geq1}
\bigl(g_{r+jm,i}-g_{r+m+jm,i+1}\bigr)\\
&=
\sum_{s<r+m}(v_{s,i}+u_{s-m,i})
-\sum_{s<r}(v_{s,i}+u_{s,i})+\sum_{j\geq1}
\bigl(v_{r+jm,i}-v_{r+(j+1)m,i}\bigr)\\
&=
\sum_{z=r}^{r+m-1}v_{z,i}
+\sum_{j\geq1}
\bigl(v_{r+jm,i}-v_{r+(j+1)m,i}\bigr)\\
&=
\sum_{z=r}^{r+m-1}v_{z,i}+v_{r+m,i}
=
\sum_{z=r}^{r+m}v_{z,i}
=\delta_{\mathfrak p}(U_{r,i}).
\end{align*}
Here the $u$-terms cancel after shifting their indices, and
the remaining tail of $v$-terms telescopes to $v_{r+m,i}$.

A letter $V_{r,i}$ goes from $(r,i)$ to $(r,i+1)$. 
\begin{align*}
G_{r,i+1}-G_{r,i}
&=
\sum_{s<r}\bigl(g_{s,i+1}-g_{s,i}\bigr)
-\sum_{j\geq1}
\bigl(g_{r+jm,i+1}-g_{r+jm,i}\bigr)\\
&=
\sum_{s<r}(u_{s-m,i}-u_{s,i})
-\sum_{j\geq1}
\bigl(u_{r+(j-1)m,i}-u_{r+jm,i}\bigr)
&&\text{(By Lemma~\ref{lemmabalance})}\\
&=
-\sum_{z=r-m}^{r-1}u_{z,i}-u_{r,i} =
-\sum_{z=r-m}^{r}u_{z,i} =
\delta_{\mathfrak p}(V_{r,i}).
\end{align*}
The two tails of $u$-terms cancel except at levels
$r-m,\ldots,r$.

Finally, $E_r$ goes from $(r,L+1)$ to $(r-n,1)$.
\begin{align*}
G_{r-n,1}-G_{r,L+1}
&=
\sum_{s<r-n}g_{s,1}
-\sum_{j\geq1}g_{r-n+jm,1}
-G_{r,L+1}\\
&=
\sum_{s<r}g_{s,L+1}
-\sum_{j\geq1}g_{r+jm,L+1}
-G_{r,L+1}
&&\text{(By \eqref{eqbal2})}\\
&=0 =\delta_{\mathfrak p}(E_r).
\end{align*}

These three calculations prove the last equality.
\end{proof}

Since these differences of $G$ telescope along a walk, they sum to zero
around a closed walk. This gives the following exchange lemma.

\begin{lemma}\label{lemmaExchange}
Suppose $A$ and $B$ are walks from a state to a
common terminal state, and $C$ is a walk from that terminal state
back to their initial state. In any proper word containing $ACB$ as consecutive
factors, replacing $ACB$ by $BCA$ gives a proper word with
the same profile, initial state, and value of $J$.
\end{lemma}

\begin{proof}
Both $AC$ and $BC$ are closed walks, so the replacement gives a
closed walk with the same profile and initial state; it is a proper word
by the converse construction in Section~\ref{subsec:words}. Pairs
involving a letter
outside the replacement contribute equally to $J$ in both
words. For the pairs inside the replacement,
\[
\begin{aligned}
J(ACB)-J(BCA)
&=
\sum_{\substack{e\in AC\\ f\in BC}}
\bigl(\kappa(e,f)-\kappa(f,e)\bigr)\\
&=
\sum_{e\in AC}
\bigl(G_{\text{departure state of }e}
      -G_{\text{arrival state of }e}\bigr)
=0.
\end{aligned}
\]
The first equality groups the pairs whose order changes; terms
with both letters in $C$ cancel in pairs. For each letter
$e$ of $AC$, the inner sum over $BC$ is the negative
of $\delta_{\mathfrak p}(e)$ for the profile $\mathfrak p$ of $BC$. This profile
is balanced because $BC$ is closed, so Lemma~\ref{lemmaDeltaG}
gives the second equality, with $G$ computed from the counts
of $BC$. Its state differences sum to zero because
$AC$ is closed. The argument also covers an empty $C$.
\end{proof}

\begin{proposition}\label{propstep1}
Fix a realizable profile $\mathfrak{p}$. For each $a$ with
$\mathfrak{W}_a(\mathfrak{p})\neq\varnothing$, there is a quantity
$\eta_a$ such that
\[
h_\sigma(w)-I(w)=\eta_a
\qquad
\text{for every }w\in\mathfrak{W}_a(\mathfrak{p}).
\]
In particular, this difference depends on the fixed profile and the
starting level $a$, but not on the order of the letters in $w$.
\end{proposition}

\begin{proof}
Take two proper words with the stated profile and initial state.
We use Lemma~\ref{lemmaExchange} to increase their common prefix
until the words agree.

Remove their longest common prefix. Let
$e$ be the next letter of the second word. Since the
profiles agree, an occurrence of $e$ appears later in
the first word. The letters preceding that occurrence,
after the common prefix, form a closed segment at the
current state: the segment ends where $e$ departs.

Follow the first word after $e$ until it first reaches a
state visited by this closed segment. Such a meeting must
occur. Otherwise every letter after $e$ in the first word departs from
a state that the segment does not visit and arrives at such a state.
After the common prefix and $e$, the second word still has to use the
letters of the segment and the letters after $e$ in the first word. It
stands at the arrival state of $e$, which the segment does not visit,
so it can only use letters that come after $e$ in the first word, and
these never lead back to a state of the segment. Thus the letters of the segment, which is not
empty, would never be used. Split the segment as $AC$ at that meeting
state, and let $B$ run from $e$ to the meeting.
Then $A$ and $B$ have the same departure and arrival
states, and $C$ returns to their common departure state.
By Lemma~\ref{lemmaExchange}, the exchange
$ACB\mapsto BCA$ preserves the profile and $J$, while
extending the common prefix.
Repeating the exchange transforms the first word into
the second.
Thus $J$ is constant on
$\mathfrak W_a(\mathfrak p)$, and we set
$\eta_a:=J(w)$ for any $w\in\mathfrak W_a(\mathfrak p)$.
\end{proof}

\begin{remark}\label{remark1}
Suppose $\nu_a>0$. Then a word in
$\mathfrak W_0(\mathfrak p)$ visits $(a,1)$, and rotating it to start
at that state gives a word in $\mathfrak W_a(\mathfrak p)$, by the converse
construction in Section~\ref{subsec:words}. Perform the rotation one letter
at a time, moving the first letter to the end at each step. By
Lemma~\ref{lemmaDeltaG}, each move changes $J$ by the difference between
$G$ at the arrival and departure states of the moved letter. The moved
letters form a walk from $(0,1)$ to $(a,1)$, so these differences
telescope, giving
\[
\eta_a-\eta_0=G_{a,1}-G_{0,1}.
\]
\end{remark}

Proposition~\ref{propstep1} fixes $J=h_\sigma-I$ at each initial
state. To count $h_\sigma$, it remains to enumerate $I$, beginning
with the order of departures at individual states.

At a state $(r,i)$, $1\leq i\leq L$, there are $u_{r,i}$
letters $U_{r,i}$ and $v_{r,i}$ letters
$V_{r,i}$. The order in which these letters occur in a proper word is
the \emph{queue} at that state. Set
\[
I_{r,i}(w):=
\#\left\{
U_{r,i}\text{ occurring before }V_{r,i}
\text{ in this queue}
\right\}.
\]
Every pair counted by $I(w)$ occurs in one queue, so
\[
I(w)=\sum_{r,i}I_{r,i}(w).
\]

If the last departure is $U_{r,i}$, removing it does not
change $I_{r,i}$. The remaining queue has $u_{r,i}-1$
letters $U_{r,i}$ and $v_{r,i}$ letters $V_{r,i}$, so
\[
\sum_{\text{last letter }U_{r,i}}q^{I_{r,i}}
=
\begin{bmatrix}
u_{r,i}+v_{r,i}-1\\u_{r,i}-1
\end{bmatrix}_q
=
\frac{[g_{r,i}-1]_q!}
{[u_{r,i}]_q![v_{r,i}]_q!}\,[u_{r,i}]_q.
\]
If the last departure is $V_{r,i}$, all $u_{r,i}$
letters $U_{r,i}$ precede it. Removing it decreases
$I_{r,i}$ by $u_{r,i}$, and the remaining queue has
$u_{r,i}$ letters $U_{r,i}$ and $v_{r,i}-1$ letters
$V_{r,i}$. Hence
\[
\sum_{\text{last letter }V_{r,i}}q^{I_{r,i}}
=
q^{u_{r,i}}
\begin{bmatrix}
u_{r,i}+v_{r,i}-1\\u_{r,i}
\end{bmatrix}_q
=
\frac{[g_{r,i}-1]_q!}
{[u_{r,i}]_q![v_{r,i}]_q!}\,
q^{u_{r,i}}[v_{r,i}]_q.
\]
For $g_{r,i}>0$, write the common factor as
\[
\gamma_{r,i}(q):=
\frac{[g_{r,i}-1]_q!}
{[u_{r,i}]_q![v_{r,i}]_q!}.
\]
The remaining factors, $[u_{r,i}]_q$ and
$q^{u_{r,i}}[v_{r,i}]_q$, depend only on the type of
the last departure.

At $(r,L+1)$, all $w_r$ departures have type $E_r$,
so there is one queue of weight $1$. For $w_r>0$,
we factor its weight in the same way as the other queues:
\[
1=\frac{1}{[w_r]_q}\,[w_r]_q.
\]
The factor $1/[w_r]_q$ depends only on the profile,
while $[w_r]_q$ becomes the weight of its last departure.
Set
\[
\Gamma_{\mathfrak p}(q):=
\prod_{\substack{r,\ 1\leq i\leq L\\g_{r,i}>0}}
\gamma_{r,i}(q),
\qquad
D_{\mathfrak p}(q):=
\prod_{\substack{r\\w_r>0}}[w_r]_q.
\]
These products depend only on the profile. Their quotient
collects the factors from the queues that do not depend
on the type of the last departure.

The last departures cannot be chosen independently.
Following them from a visited state must eventually reach
the departure state of the final letter. Fix a state $z$.
At each visited state
$x\neq z$, select one departing letter type so that
following the selected transitions from any such state
eventually reaches $z$. These transitions form a spanning tree of the
visited states in which every edge points toward $z$; we call it a
\emph{tree rooted at} $z$. Such trees appear in the BEST theorem of
van Aardenne-Ehrenfest and de Bruijn~\cite{AEdB51} and Tutte and
Smith~\cite{TS41} (see also Stanley~\cite[Section~5.6]{EC2}), and
Lemma~\ref{lemmaQueueTree} below is a $q$-weighted form of that theorem.
Assign the weights
\[
\operatorname{wt}(U_{r,i}):=[u_{r,i}]_q,\qquad
\operatorname{wt}(V_{r,i}):=
q^{u_{r,i}}[v_{r,i}]_q,\qquad
\operatorname{wt}(E_r):=[w_r]_q,
\]
and define
\[
T_z(q):=
\sum_{\substack{\text{trees}\\\text{rooted at }z}}
\ \prod_{\text{selected letters }e}\operatorname{wt}(e).
\]

\begin{example}\label{exampleTreeProfile}
Continue with the profile $\mathfrak p$ of
Example~\ref{exampleRowWords} and fix the initial state $(4,1)$.
The only states with both $U$- and $V$-departures are
$(4,2)$ and $(5,8)$. In the word of
Example~\ref{exampleRowWords}, their last departures are
$V_{4,2}$ and $U_{5,8}$, respectively. Thus their queues
are $U_{4,2}V_{4,2}$ and $V_{5,8}U_{5,8}$.

Keep $U_{5,8}$ as the last departure at $(5,8)$, but choose
$U_{4,2}$ as the last departure at $(4,2)$. Its queue is now
$V_{4,2}U_{4,2}$. Reading the queues from $(4,1)$ gives the
word on the left of Figure~\ref{figAlternativeWord}. The first
line of this word corresponds to the top row of the RST in the middle
and to the first path column on the right. The remaining lines
correspond in the same order.

\begin{center}
\begin{minipage}{\linewidth}
\centering
\begin{tikzpicture}[x=5.5mm,y=5.5mm]
  \node[anchor=center,align=center,font=\scriptsize,inner sep=0pt] at (3.5,7.5)
    {$\begin{gathered}V_{4,1}\,V_{4,2}\,V_{4,3}\,V_{4,4}\,V_{4,5}\\[-2pt]
     V_{4,6}\,V_{4,7}\,V_{4,8}\,E_4\end{gathered}$};
  \node[anchor=center,align=center,font=\scriptsize,inner sep=0pt] at (3.5,6.1)
    {$\begin{gathered}V_{0,1}\,V_{0,2}\,V_{0,3}\,U_{0,4}\,U_{3,5}\\[-2pt]
     V_{6,6}\,U_{6,7}\,V_{9,8}\,E_9\end{gathered}$};
  \node[anchor=center,align=center,font=\scriptsize,inner sep=0pt] at (3.5,4.7)
    {$\begin{gathered}V_{5,1}\,V_{5,2}\,V_{5,3}\,V_{5,4}\,V_{5,5}\\[-2pt]
     V_{5,6}\,V_{5,7}\,V_{5,8}\,E_5\end{gathered}$};
  \node[anchor=center,align=center,font=\scriptsize,inner sep=0pt] at (3.5,3.3)
    {$\begin{gathered}U_{1,1}\,U_{4,2}\,V_{7,3}\,V_{7,4}\,V_{7,5}\\[-2pt]
     U_{7,6}\,V_{10,7}\,V_{10,8}\,E_{10}\end{gathered}$};
  \node[anchor=center,align=center,font=\scriptsize,inner sep=0pt] at (3.5,1.9)
    {$\begin{gathered}V_{6,1}\,V_{6,2}\,U_{6,3}\,V_{9,4}\,V_{9,5}\\[-2pt]
     V_{9,6}\,V_{9,7}\,V_{9,8}\,E_9\end{gathered}$};
  \node[anchor=center,align=center,font=\scriptsize,inner sep=0pt] at (3.5,.5)
    {$\begin{gathered}V_{5,1}\,V_{5,2}\,V_{5,3}\,V_{5,4}\,V_{5,5}\\[-2pt]
     V_{5,6}\,V_{5,7}\,U_{5,8}\,E_8\end{gathered}$};
\node[font=\small] at (3.5,-.75) {word};

  \draw[blue!55,densely dashed,line width=.4pt]
    (12,-.35)--(12,8.35);
  \node[blue!70,font=\scriptsize,anchor=south]
    at (12,8.43) {$x=0$};
  \def\RBox#1#2#3{%
    \draw ({11+#2},#1) rectangle ({12+#2},{#1+1});%
    \node[font=\scriptsize] at ({11.5+#2},{#1+.5}) {$#3$};}
  \def\VLeft#1#2{%
    \node[font=\tiny,anchor=center,inner sep=0pt] at ({11.5+#2},{#1+.5}) {$-\infty$};}
  \def\VRight#1#2{%
    \node[font=\tiny,anchor=center,inner sep=0pt] at ({11.5+#2},{#1+.5}) {$+\infty$};}
  \VLeft{7}{0}\VRight{7}{1}
  \draw (12,7)--(12,8);
  \VLeft{5.6}{-2}\RBox{5.6}{-1}{4}
  \RBox{5.6}{0}{5}\RBox{5.6}{1}{7}\VRight{5.6}{2}
  \VLeft{4.2}{0}\VRight{4.2}{1}
  \draw (12,4.2)--(12,5.2);
  \VLeft{2.8}{-1}\RBox{2.8}{0}{1}
  \RBox{2.8}{1}{2}\RBox{2.8}{2}{6}\VRight{2.8}{3}
  \VLeft{1.4}{0}\RBox{1.4}{1}{3}\VRight{1.4}{2}
  \VLeft{0}{0}\RBox{0}{1}{8}\VRight{0}{2}
  \node[font=\small] at (11.5,-.75) {RST};

  \begin{scope}[shift={(17,0)}]
    \draw[step=1,gray!35,very thin] (0,0) grid (6,8);
    \draw[blue,dashed,line width=.85pt] (0,0)--(6,8);
    \draw[blue!65,densely dashed,line width=.7pt] (1,0)--(6,6.6667);
    \draw[line width=1pt]
      (0,0)--(1,0)--(1,3)--(2,3)--(3,3)
      --(3,6)--(4,6)--(4,7)--(5,7)--(5,8)--(6,8);
    \foreach \x/\y/\z in
      {1/.5/4,1/1.5/5,1/2.5/7,
       3/3.5/1,3/4.5/2,3/5.5/6,4/6.5/3,5/7.5/8}
      \node[font=\scriptsize,fill=white,inner sep=.6pt] at ({\x+.5},\y) {$\z$};
  \end{scope}
  \node[font=\small] at (20,-.75) {preference function};
\end{tikzpicture}

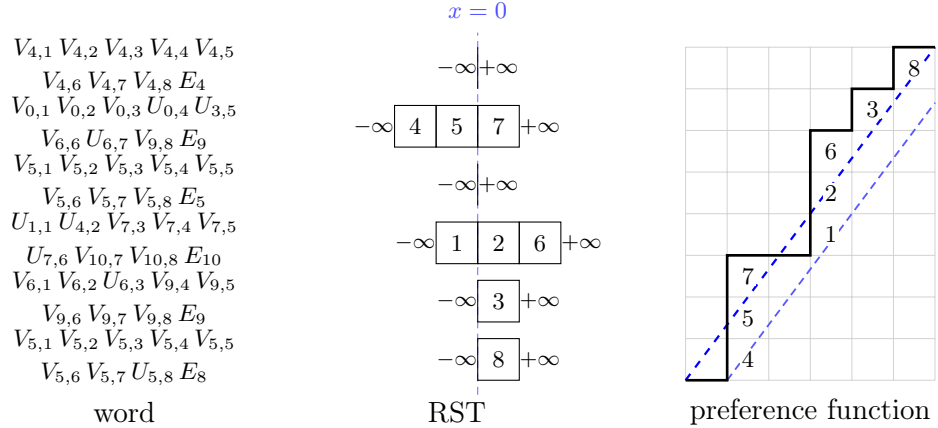
\captionof{figure}{A second $(6,8)$-preference function, its RST, and its word}\label{figAlternativeWord}
\end{minipage}
\end{center}

The new preference function has vertical run lengths
$(0,3,0,3,1,1)$. From left to right, its first three columns
carry $\varnothing,(4,5,7),\varnothing$, and its last three
carry $(1,2,6),(3),(8)$.
In particular, the first line of the word has no $U$, matching the empty
top RST row and the empty run in column $x=0$. The fourth line of the
word contains $U_{1,1},U_{4,2},U_{7,6}$, matching the entries
$1,2,6$ in the fourth RST row and the run in column $x=3$.
This is again a $(6,8)$-preference function.

The root is $(a+n,L+1)=(8,9)$. A tree rooted there must select
$U_{5,8}$ at $(5,8)$; at $(4,2)$ it may select either
$U_{4,2}$ or $V_{4,2}$, of weights $1$ and $q$.
The seven states $(5,i)$, $1\leq i\leq7$, together with
$(9,8)$, have two $V$-departures each; $(9,9)$ has two
$E_9$-departures. Each of these nine states supplies a
factor $[2]_q$.
Consequently,
\[
T_{8,9}(q)=[2]_q^9(1+q)=[2]_q^{10}.
\]
Reading the queues for these two choices gives the word of
Example~\ref{exampleRowWords} and the word in
Figure~\ref{figAlternativeWord}; choosing $V_{5,8}$ last
does not give a tree rooted at $(8,9)$. Thus these are exactly
the two proper words with this profile and initial state.
The new word has $I=0$ and
$h_\sigma=4$; the word of Example~\ref{exampleRowWords}
has $I=1$ and $h_\sigma=5$. Hence
\[
\sum_{w\in\mathfrak W_4(\mathfrak p)}q^{I(w)}
=1+q,
\qquad
\phi_4(\mathfrak p;q)=q^4+q^5.
\]
\end{example}

\begin{lemma}\label{lemmaQueueTree}
If $\nu_a>0$, then
\begin{equation}\label{eqle1}
\sum_{w\in\mathfrak W_a(\mathfrak p)}q^{I(w)}
=
\Gamma_{\mathfrak p}(q)\,
\frac{[\nu_a]_q}{D_{\mathfrak p}(q)}\,
T_{a+n,L+1}(q).
\end{equation}
\end{lemma}

\begin{proof}
For $w\in\mathfrak W_a(\mathfrak p)$, let
$z=(a+n,L+1)$. The final letter is $E_{a+n}$.
At every other visited state, select the type of its
last departure in $w$. The selected transitions form
a tree rooted at $z$. Indeed, their departure times
strictly increase along any directed path, so they cannot
form a cycle. Since every other state has a selected
departure, each such path ends at $z$.

Conversely, choose such a tree and, at every state other
than $z$, a queue whose last letter has the selected
type. Set aside one $E_{a+n}$ as the final letter,
and put the remaining copies in the queue at $z$.
Starting at $(a,1)$, repeatedly take the next unused
letter from the queue at the current state.

We now check that this reading uses every letter before
the final $E_{a+n}$. Once that letter is set aside,
the counts have one extra departure at $(a,1)$ and
one extra arrival at $z$. Hence the reading can stop
only at $z$. If letters remained unused there, their
arrivals and departures would balance at every state.
A state with an unused departure would still have its
selected last departure unused. Following the selected
transitions would then force an unused departure at $z$,
contradicting the fact that the reading stopped at $z$. Thus every letter
is used, and appending $E_{a+n}$ gives a closed walk with profile
$\mathfrak p$ that begins at $(a,1)$; it is a proper word by the converse
construction in Section~\ref{subsec:words}.

The common factors of the queues contribute
$\Gamma_{\mathfrak p}(q)/D_{\mathfrak p}(q)$.
The selected last departures contribute the tree weight,
while the final $E_{a+n}$ contributes
$[w_{a+n}]_q=[\nu_a]_q$. Summing over the trees gives
\eqref{eqle1}.
\end{proof}

\begin{lemma}\label{lemmaTreeRecurrence}
For each visited state $(r,i+1)$, with $1\leq i\leq L$,
\begin{equation}\label{rec1}
[g_{r,i+1}]_qT_{r,i+1}(q)
=
[u_{r-m,i}]_qT_{r-m,i}(q)
+
q^{u_{r,i}}[v_{r,i}]_qT_{r,i}(q).
\end{equation}
Terms with zero letter count are omitted. At a visited state
$(r,1)$,
\begin{equation}\label{rec2}
[g_{r,1}]_qT_{r,1}(q)
=
[w_{r+n}]_qT_{r+n,L+1}(q).
\end{equation}
In particular, if $\nu_r>0$, then
$T_{r,1}(q)=T_{r+n,L+1}(q)$.
\end{lemma}

\begin{proof}
For a visited state $y$, let $\deg^+(y)$ be the sum of
the weights of the letter types departing from $y$.
Then we claim that
\[
\deg^+(y)\,T_y(q)
=
\sum_{\substack{f\text{ a letter type}\\x\longrightarrow y}}
\operatorname{wt}(f)\,T_x(q).
\]
On the left, choose a tree rooted at $y$ and a letter
type $e$ departing from $y$. Adding $e$ creates
one directed cycle. Let $f:x\to y$ be the transition
on that cycle that arrives at $y$. Removing $f$
makes $x$ the new root, with $f$ as the chosen
letter type on the right. Conversely, add $f$ to a
tree rooted at $x$, then remove the transition on the
resulting cycle that departs from $y$. The product of
the selected weights is preserved, proving the claim.

Take $y=(r,i+1)$, where $1\leq i\leq L$.
The letters arriving at $y$ are $U_{r-m,i}$
and $V_{r,i}$, with weights
$[u_{r-m,i}]_q$ and $q^{u_{r,i}}[v_{r,i}]_q$, respectively.
When $i+1\leq L$, the total weight of types departing
from $y$ is
\[
[u_{r,i+1}]_q+
q^{u_{r,i+1}}[v_{r,i+1}]_q
=[g_{r,i+1}]_q.
\]
When $i=L$, only $E_r$ departs from $(r,L+1)$;
its weight is
$[w_r]_q=[g_{r,L+1}]_q$. Thus, in both cases,
we obtain \eqref{rec1}. Only
$E_{r+n}$ arrives at $(r,1)$, which gives \eqref{rec2}.
When $\nu_r=g_{r,1}=w_{r+n}>0$, we may cancel the common nonzero
factor to obtain $T_{r,1}(q)=T_{r+n,L+1}(q)$.
\end{proof}

\begin{lemma}\label{lemmaTreeRoots}
At every visited state,
\begin{equation}\label{eqclaim}
T_{r,i}(q)
=
K_{\mathfrak p}(q)\,
q^{\sum_{j\geq1}g_{r+jm,i}}, \text{ where }
K_{\mathfrak p}(q)
:=
T_{0,1}(q)\,
q^{-\sum_{j\geq1}g_{jm,1}}.
\end{equation}
In particular, whenever $\nu_a>0$,
\[
T_{a+n,L+1}(q)
=K_{\mathfrak p}(q)\,
q^{\sum_{j\geq1}g_{a+jm,1}}.
\]
\end{lemma}

\begin{proof}
The state $(0,1)$ is visited because $\nu_0=\ell>0$,
so the value there fixes the common factor. We first
verify that the proposed expression satisfies
\eqref{rec1}. Shifting the index gives
\[
\sum_{j\geq1}g_{r-m+jm,i}
=
g_{r,i}+\sum_{j\geq1}g_{r+jm,i}.
\]
Lemma~\ref{lemmabalance}, applied at each $r+jm$, gives
\[
\sum_{j\geq1}g_{r+jm,i+1}
=
\sum_{j\geq1}
\bigl(v_{r+jm,i}+u_{r+(j-1)m,i}\bigr)
=
u_{r,i}+\sum_{j\geq1}g_{r+jm,i}.
\]
Substitute these powers into the right side of
\eqref{rec1} and cancel $K_{\mathfrak p}(q)$. The result is
\[
\begin{aligned}
&[u_{r-m,i}]_q
q^{\sum_{j\geq1}g_{r-m+jm,i}}
+
q^{u_{r,i}}[v_{r,i}]_q
q^{\sum_{j\geq1}g_{r+jm,i}}\\
&\qquad=
q^{\sum_{j\geq1}g_{r+jm,i+1}}
\bigl(q^{v_{r,i}}[u_{r-m,i}]_q+[v_{r,i}]_q\bigr)\\
&\qquad=
q^{\sum_{j\geq1}g_{r+jm,i+1}}
[g_{r,i+1}]_q.
\end{aligned}
\]
Here the last equality follows from
$g_{r,i+1}=v_{r,i}+u_{r-m,i}$ and
$[v]_q+q^v[u]_q=[u+v]_q$. This is the
left side of \eqref{rec1} with the proposed
expression substituted.

For the transition between rows, \eqref{eqbal2} gives
\[
g_{r+jm,1}=g_{r+n+jm,L+1}
\qquad(j\geq1).
\]
Thus the powers at $(r,1)$ and $(r+n,L+1)$ agree,
as required by \eqref{rec2}.

It remains to show that the
recurrences allow only one common factor across all visited
states. Fix $q>0$; every letter type that occurs has positive weight.
Divide the recurrence for $T$ by the recurrence
just verified for the proposed powers of $q$.
At each state, the quotient of $T_{r,i}(q)$ by
$q^{\sum_{j\geq1}g_{r+jm,i}}$ is a weighted average
of the corresponding quotients at states from which
a letter arrives.

Choose a state where this quotient is maximal. Since
all weights are positive, every state with a transition
into it has the same quotient. A proper word with the
realizable profile is a closed walk through every
visited state. Following this walk backward reaches
all of them, so the same value propagates to every
state. At $(0,1)$ it is $K_{\mathfrak p}(q)$.
This proves \eqref{eqclaim} for all $q>0$, and
therefore as an identity in $q$.
For $\nu_a>0$, \eqref{rec2} gives
$T_{a+n,L+1}(q)=T_{a,1}(q)$. The stated expression follows
from \eqref{eqclaim} at $(a,1)$.
\end{proof}

\begin{example}\label{exampleTreeRootCheck}
For the profile of Example~\ref{exampleTreeProfile}, a tree rooted
at $(0,1)$ must select $V_{4,2}$ at $(4,2)$, of weight $q$.
At $(5,8)$, either $U_{5,8}$ or $V_{5,8}$ may be selected,
of weights $1$ and $q$. The other nine nonunit factors are
the $[2]_q$ factors identified in that example. Direct counting
therefore gives
\[
T_{0,1}(q)=q[2]_q^9(1+q)=q[2]_q^{10}.
\]
Here $\sum_{j\geq1}g_{3j,1}=\nu_6=1$, whereas
$\sum_{j\geq1}g_{4+3j,1}=0$. Thus Lemma~\ref{lemmaTreeRoots}
gives $K_{\mathfrak p}(q)=q^{-1}T_{0,1}(q)=[2]_q^{10}$ and
$T_{8,9}(q)=[2]_q^{10}$, in agreement with the direct tree count
in Example~\ref{exampleTreeProfile}.

Moreover, $\Gamma_{\mathfrak p}(q)=[2]_q^{-8}$,
$D_{\mathfrak p}(q)=[2]_q$, and $\nu_4=1$. In this case,
Lemma~\ref{lemmaQueueTree} gives
\[
\sum_{w\in\mathfrak W_4(\mathfrak p)}q^{I(w)}
=[2]_q^{-8}\frac{1}{[2]_q}[2]_q^{10}
=[2]_q=1+q,
\]
as found in Example~\ref{exampleTreeProfile}.
\end{example}

\begin{proposition}\label{propstep2}
For every $a\geq0$ and every realizable profile $\mathfrak{p}$,
\[
\sum_{w\in\mathfrak{W}_a(\mathfrak{p})}q^{I(w)}
=
C_{\mathfrak{p}}(q)\,
q^{\sum_{j\geq1}g_{a+jm,1}}[\nu_a]_q,
\]
where $C_{\mathfrak{p}}(q)$ is a rational function of $q$ that does not
depend on $a$.
\end{proposition}

\begin{proof}
If $\nu_a=0$, there is no word starting at $(a,1)$,
and both sides vanish. Suppose $\nu_a>0$. Combining
Lemmas~\ref{lemmaQueueTree} and \ref{lemmaTreeRoots} gives
\[
\sum_{w\in\mathfrak{W}_a(\mathfrak{p})}q^{I(w)}
=
\Gamma_{\mathfrak p}(q)\,
\frac{[\nu_a]_q}{D_{\mathfrak p}(q)}\,
K_{\mathfrak p}(q)\,
q^{\sum_{j\geq1}g_{a+jm,1}} =
C_{\mathfrak p}(q)\,
q^{\sum_{j\geq1}g_{a+jm,1}}[\nu_a]_q,
\]
where
\[
C_{\mathfrak p}(q)
:=
\frac{\Gamma_{\mathfrak p}(q)K_{\mathfrak p}(q)}
{D_{\mathfrak p}(q)}.
\]
Every factor in $C_{\mathfrak p}(q)$ depends
only on the profile.
\end{proof}

\begin{proof}[Proof of Proposition~\ref{propqcyc}]
Fix a realizable profile $\mathfrak p$. The assertion
is immediate when $\nu_a=0$, since
$\mathfrak W_a(\mathfrak p)=\varnothing$.
Assume $\nu_a>0$.

Proposition~\ref{propstep1} separates the order-dependent
part of $h_\sigma$ from $I$, and Proposition~\ref{propstep2}
enumerates the latter. Their two exponents have a simple
combined dependence on $a$. By Remark~\ref{remark1}
and the definition of $G$,
\begin{align*}
\eta_a+\sum_{j\geq1}g_{a+jm,1}
&=
\eta_0+G_{a,1}-G_{0,1}
+\sum_{j\geq1}g_{a+jm,1}\\
&=
\eta_0+\sum_{r<a}g_{r,1}
+\sum_{j\geq1}g_{jm,1}\\
&=
P_a+\eta_0+\sum_{j\geq1}g_{jm,1}.
\end{align*}
The last equality uses $g_{r,1}=\nu_r$ and the
convention that counts at negative levels vanish.
In particular, the contribution from moving the
initial state and the tail sum in Proposition~\ref{propstep2}
combine to give $P_a$.

Propositions~\ref{propstep1} and \ref{propstep2},
together with $\nu_0=\ell$, now give
\begin{align*}
[\ell]_q\phi_a(\mathfrak p;q)
&=
[\nu_0]_q\,C_{\mathfrak p}(q)\,
q^{\eta_a+\sum_{j\geq1}g_{a+jm,1}}[\nu_a]_q\\
&=
q^{P_a}[\nu_a]_q\,C_{\mathfrak p}(q)\,
q^{\eta_0+\sum_{j\geq1}g_{jm,1}}[\nu_0]_q\\
&=
q^{P_a}[\nu_a]_q\phi_0(\mathfrak p;q),
\end{align*}
as required.
\end{proof}

\subsection{Proof of Theorem~\ref{thmDelta}}\label{subsec:delta}
The proof of Iraci, Pagaria, and Paolini~\cite[Theorem~1.2, Section~5.2]{IPP26}
contains a combinatorial skewing argument for weighted preference
function sums. Applying it to Theorem~\ref{thmmain} gives the
following identity in our operator notation.

\begin{corollary}\label{cor:fall-decorated}
Let $k,m,n$ be positive integers with $\gcd(m,n)=1$, and let
$r\geq0$. Then
\begin{equation}\label{eq:fall-decorated}
(-1)^{k-1}\frac{[km]_q}{[k]_q}\,
s_{(km-1)^r}^{\perp}
\bigl(p_k[-MX^{m,n+rm}]\cdot1\bigr)
=
\sum_{\pi\in\operatorname{LRP}(km+r,kn+r)_{\ast r}}
q^{\dinv(\pi)}t^{\area(\pi)}x_\pi.
\end{equation}
\end{corollary}
\begin{proof}
The argument of Iraci, Pagaria, and Paolini~\cite[Propositions~4.5, 4.10, and 4.15;
Theorem~5.1 and Section~5.2]{IPP26} gives the following
identity directly for weighted path sums:
\[
s_{(km-1)^r}^{\perp}
\sum_{\rho\in\Pref_{km,k(n+rm)}}
q^{\dinv(\rho)}t^{\area(\rho)}x_\rho
=
\sum_{\pi\in\operatorname{LRP}(km+r,kn+r)_{\ast r}}
q^{\dinv(\pi)}t^{\area(\pi)}x_\pi.
\]
The skewing operator is applied here to a symmetric function. Indeed,
the area does not depend on the labels, and by
Propositions~\ref{propNsym} and~\ref{propdinvh}, the sum of
$q^{\dinv(\rho)}x_\rho$ over the labelings $\rho$ of a fixed path is
$N^\sigma_{\beta/\alpha}(X;q)$, with $\alpha,\beta,\sigma$ as in
Definition~\ref{defalpha} for that path. Schur skewing first expands the
left side into signed labelings.
The sign-reversing involution in that argument cancels the terms
without the required fall decorations; its surviving terms preserve
the area and dinv weights. Without decorations, the area and dinv
of Iraci, Pagaria, and Paolini agree with those of
Section~\ref{subsec:comb} by \cite[Remarks~3.5 and 3.22]{IPP26}.
Since $\gcd(m,n+rm)=1$, Theorem~\ref{thmmain} identifies the sum
on the left with
$(-1)^{k-1}[km]_q/[k]_q\,
p_k[-MX^{m,n+rm}]\cdot1$, proving \eqref{eq:fall-decorated}.
\end{proof}

The following proposition supplies the operator identity needed for
Theorem~\ref{thmDelta}. Iraci, Pagaria, and Paolini~\cite[Conjecture~6.10]{IPP26}
proposed the identity
\[
\Theta_{e_d}\nabla\omega(p_k)=s^\perp_{(k-1)^d}\,p_{k,k(d+1)},
\]
where we write $k$ and $d$ for their $n$ and $k$. If $p_{k,k(d+1)}$ is
normalized as in footnote~\ref{fnsign}, so that
$p_{k,k(d+1)}=p_k[-MX^{1,d+1}]\cdot1$, then the left side should be
$\Theta_{e_d}\nabla p_k=(-1)^{k-1}\Theta_{e_d}\nabla\omega(p_k)$.

\begin{proposition}\label{propDelta}
For $k\geq1$ and $d\geq0$,
\[
\Theta_{e_d}\nabla p_k
=
s_{(k-1)^d}^{\perp}
\bigl(p_k[-MX^{1,d+1}]\cdot1\bigr).
\]
\end{proposition}
\begin{proof}
We first show that, for $0\leq i<k$,
\begin{equation}\label{eqThetaGamma}
\Theta_{e_d}\bigl(D_{\gamma^{(i)}}\cdot1\bigr)
=
\sum_{\beta\in\mathbb Z_{\geq0}^k, |\beta|=d}
D_{\gamma^{(i)}+\beta}\cdot1.
\end{equation}
For $i=0$, the word of $\gamma^{(0)}=(1^k)$ is
$k(k-1)\cdots1$. Inserting $u_j$ copies of the new largest
letter after its $j$th letter, where $u_1+\cdots+u_k=d$,
gives descent composition $\gamma^{(0)}+(u_1,\ldots,u_k)$.
Theorem~\ref{thmTheta} therefore proves \eqref{eqThetaGamma}
in this case.

For $1\leq i<k$, omit the final zero from $\gamma^{(i)}$;
this does not change its action on $1$ by
Blasiak et al.~\cite[Lemma~3.6.2]{path}. The associated word has $k$ letters
in decreasing order, except that its $i$th and $(i+1)$st
letters are equal. Again let $u_j$ count the new letters
inserted after the $j$th old letter. If $u_i=0$, the two
equal letters remain in one weakly increasing block and
the descent composition is $\gamma^{(i)}+\beta$ with
\[
\beta=(u_1,\ldots,u_{i-1},u_{i+1},\ldots,u_k,0).
\]
The final zero may be omitted when $D_{\gamma^{(i)}+\beta}$
acts on $1$. If $u_i>0$, the insertion separates the equal
letters and gives the same form with
\[
\beta=(u_1,\ldots,u_{i-1},u_i-1,
       u_{i+1},\ldots,u_{k-1},u_k+1).
\]
The two cases give respectively all $\beta$ of weight $d$
with $\beta_k=0$ and all those with $\beta_k>0$, each
exactly once. Theorem~\ref{thmTheta} now proves
\eqref{eqThetaGamma} for every $i$.

By Lemma~\ref{lemma:nabla},
$p_k[-MX^{1,1}]\cdot1=\nabla p_k$.
Apply $\Theta_{e_d}$ to the case $d=0$ of \eqref{eqnegut}, then use
\eqref{eqThetaGamma}, Proposition~\ref{propnegut}, and \eqref{eqnegut}
for the given $d$:
\begin{align*}
(-1)^{k-1}\Theta_{e_d}\nabla p_k
&=
\sum_{\beta\in\mathbb Z_{\geq0}^k, |\beta|=d}
D_{\gamma^{(0)}+\beta}\cdot1
+\sum_{i=1}^{k-1}(qt)^{k-i}
\sum_{\beta\in\mathbb Z_{\geq0}^k, |\beta|=d}
D_{\gamma^{(i)}+\beta}\cdot1\\
&=
s_{(k-1)^d}^{\perp}
\left(D_{\gamma^{(0)}+d\mathbf1}\cdot1
+\sum_{i=1}^{k-1}(qt)^{k-i}
D_{\gamma^{(i)}+d\mathbf1}\cdot1\right)\\
&=
(-1)^{k-1}s_{(k-1)^d}^{\perp}
\bigl(p_k[-MX^{1,d+1}]\cdot1\bigr).
\end{align*}
Canceling the common sign gives the result.
\end{proof}

For the reader's convenience, we recall the statistics on
$\operatorname{LRP}(k+d,k+d)_{\ast d}$ used in Theorem~\ref{thmDelta}; see
Iraci, Pagaria, and Paolini~\cite[Section~3 and Remark~3.22]{IPP26}.
Number the steps of a path from $1$ to $2(k+d)$. For a North step $i$
with lower endpoint $(x,y)$, and for an East step $i$ with right endpoint
$(x,y)$, set $v_i:=y-x$. Let $\mathcal V$ be the set of North steps,
$\mathcal D$ the set of decorated falls, and $\mathcal H$ the set of East
steps that are not decorated. Then
\[
\area(\pi)=\sum_{i\in\mathcal H}\Bigl(v_i-\min_{j\in\mathcal H}v_j\Bigr)
\]
and
\[
\begin{aligned}
\dinv(\pi)={}&\#\bigl\{(i,j)\in\mathcal V^2:(v_i,i)<(v_j,j)<(v_i+1,i),\ w_i<w_j\bigr\}\\
&+\#\{i\in\mathcal V:v_i<0\}-\#\{i\in\mathcal D:v_i\leq0\},
\end{aligned}
\]
where pairs are compared in lexicographic order and $w_i$ is the label of
the North step $i$. For $d=0$, these are the statistics of
Section~\ref{subsec:comb} with $m=n=1$.

\begin{proof}[Proof of Theorem~\ref{thmDelta}]
Apply Proposition~\ref{propDelta}, followed by
Corollary~\ref{cor:fall-decorated} with $(m,n)=(1,1)$
and $r=d$. Since $[km]_q/[k]_q=1$ when $m=1$, this gives
\begin{align*}
(-1)^{k-1}\Theta_{e_d}\nabla p_k
&=(-1)^{k-1}s_{(k-1)^d}^{\perp}
  \bigl(p_k[-MX^{1,d+1}]\cdot1\bigr)\\
&=\sum_{\pi\in\operatorname{LRP}(k+d,k+d)_{\ast d}}
  q^{\dinv(\pi)}t^{\area(\pi)}x_\pi,
\end{align*}
as required.
\end{proof}

\section*{Use of AI tools}
The theoretical derivations and the proofs in this paper were developed with the help of ChatGPT 5.6 Sol.  The authors studied and checked all of this material, and wrote the paper themselves.  We also checked the main identities by computer in small cases.  The authors take full responsibility for the content of the paper.

\end{document}